\documentclass[final,3p,times]{elsarticle}       % 单栏

\usepackage{lineno} %,hyperref
\modulolinenumbers[5]

\makeatletter
\def\ps@pprintTitle{%
  \let\@oddhead\@empty
  \let\@evenhead\@empty
  \let\@oddfoot\@empty
  \let\@evenfoot\@oddfoot
}
\makeatother
\usepackage[table,xcdraw,svgnames]{xcolor}
\usepackage[colorlinks]{hyperref}
\AtBeginDocument{%
  \hypersetup{
    citecolor=magenta,
    linkcolor=blue,
    urlcolor=Blue}}
\usepackage[heading=true,scheme=plain]{ctex} %%汉化英文模板 的时候   只想支持中文，标题不汉化
\usepackage{mathrsfs}
\usepackage{amsfonts}   % \mathbb{} 花体
\usepackage{amssymb}   %  \complement
\usepackage{lmodern}   % 数学公式字体
\usepackage{bm}       % 数学符号加粗 \bm
\usepackage{bbm}  % 小写字母空心 \mathbbm

\usepackage{autobreak}       %  autobreak package
\allowdisplaybreaks[4]   % 使数学公式自动分页

 \usepackage[capitalise]{cleveref}%\usepackage[nameinlink]{cleveref} %  宏包的\cref 类似于\cite引用多个文献

  \usepackage{extarrows}  % 定义箭头  \xlongequal  \xLongleftrightarrow
\usepackage{tikz}   % 绘图
\usetikzlibrary{arrows.meta}
\usetikzlibrary{snakes}
\usepackage{pgfplots}
\usetikzlibrary{patterns}
\usepackage{tkz-euclide}
\usetikzlibrary{calc}
\usetikzlibrary{intersections}
\usepackage{subfigure}       % 插图
\numberwithin{figure}{section}
\usepackage[justification=centering]{caption}  % 标题换行
 \usepackage[shortlabels]{enumitem}  % 设定排列环境的 格式, 可以缩写简略形式
\numberwithin{equation}{section} %公式按章节编号
 \allowdisplaybreaks[4]   % 使数学公式自动分页
\usepackage{amsthm}
\newtheorem{definition}{Definition} [section]            % 斜体
\newtheorem{theorem}{Theorem}[section]            % 斜体
\newtheorem{lemma}{Lemma} [section]%单独编号,与节有关

\newdefinition{corollary}{Corollary}[section]
\newtheorem{remark}{Remark}
\def\loc{{\mathrm{loc}}}

\def\dchi{\scalebox{1.2}{$\chi$}}

\def\dint{\displaystyle\int}

\DeclareMathOperator*{\bmo}{BMO}

\DeclareMathOperator*{\essinf}{ess\, inf}
\DeclareMathOperator*{\esssup}{ess\, sup}

\DeclareMathOperator*{\almostevery}{a.e.~}

\newcommand{\mathd}{\mathrm{d}}

\biboptions{numbers,sort&compress}   % 针对natbib，引用参考文献 缩写
\begin{document}

\begin{frontmatter}

\title{{\bfseries   Characterization of boundedness of some commutators of fractional maximal functions in terms of $p$-adic vector  spaces }}
%  Characterization of BMO spaces via commutators of fractional maximal operators on the $p$-adic vector  spaces

\author[mymainaddress]{J. Wu\corref{mycorrespondingauthor}}
%\address{Department of Mathematics, Mudanjiang Normal University, Mudanjiang, 157011, China}
\cortext[mycorrespondingauthor]{Corresponding author}
\ead{jl-wu@163.com}

\author[mymainaddress]{Y. Chang}
%\ead{yunpeng_chang2023@163.com}

%\author[mysecondaryaddress]{Wenjiao Zhao\corref{mycorrespondingauthor}}
%\ead{wenjiaozhao@163.com}

\address[mymainaddress]{Department of Mathematics, Mudanjiang Normal University, Mudanjiang  157011, China}
%\address[mysecondaryaddress]{School of Mathematics, Physics and Finance, Anhui Polytechnic University, Wuhu 241000, China}

\begin{abstract}

This paper gives some characterizations of the boundedness of the maximal or nonlinear commutator of the $p$-adic fractional maximal operator $ \mathcal{M}_{\alpha}^{p}$  with the symbols belong to the $p$-adic BMO spaces on (variable) Lebesgue spaces and Morrey spaces over $p$-adic field,
by which some new characterizations of  BMO functions are obtained in the $p$-adic field context.
Meanwhile, Some equivalent relations between the $p$-adic BMO norm and the  $p$-adic (variable) Lebesgue or Morrey norm are given.

\end{abstract}

\begin{keyword}
$p$-adic field \sep BMO function \sep  fractional maximal function \sep  variable exponent Lebesgue space \sep Morrey space

\MSC[2020]   42B35  %	Harmonic analysis on Euclidean spaces    Function spaces arising in harmonic analysis
              \sep 11E95  %$p$-adic theory
%             \sep 11K70   %Harmonic analysis and almost periodicity in probabilistic number theory
%                 \sep 11F85 %$p$-adic theory, local fields [See also 14G20, 22E50]
%                  \sep 26A16 %Lipschitz (Holder) classes
                   \sep 26A33  	%Fractional derivatives and integrals
%                  \sep 26D10 %Inequalities involving derivatives and differential and integral operators
%                  \sep 46E30 % 	Spaces of measurable functions
                  \sep 47G10 % Integral operators
%                  30H35 %BMO-spaces
\end{keyword}

\end{frontmatter}

%--------------
%\linenumbers    % 显示行数
%-------------

%==========================
\section{Introduction and main results}
\label{sec:introduction}

The importance of commutators is that they are not only able to produce some characterizations of function spaces \cite{janson1978mean,paluszynski1995characterization}, but also closely related to the regularity of solutions of certain partial differential equations  (PDEs) \cite{chiarenza1993w2,difazio1993interior,ragusa2004cauchy,bramanti1995commutators}. The Coifman-Rochberg-Weiss type commutator $[b, T]$ generated by the classical singular integral operator $T$ and a suitable function $b$ is defined by
\begin{align}  \label{equ:commutator-1}
 [b,T]f      & = bT(f)-T(bf).
\end{align}
A well-known result indicates that   $[b,T]$ is bounded on $L^{s}(\mathbb{R}^{n})$  for $1<s<\infty$ if and only if $b\in \bmo(\mathbb{R}^{n})$ (bounded mean oscillation function space). The sufficiency was provided  by  Coifman et al. \cite{coifman1976factorization} and the necessity was obtained by   Janson \cite{janson1978mean}.

In addition, the $p$-adic analysis has attracted much attention in the past few decades due to  its important applications in mathematical physics, science  and technology,  such as  $p$-adic pseudo-differential equations, $p$-adic harmonic analysis, $p$-adic wavelet theory,  etc (see \cite{albeverio2006harmonic,shelkovich2009padic,khrennikov2010non,torresblanca2023some}).

Denote by $\mathbb{Q}$, $\mathbb{N}$, $\mathbb{Z}$  and $\mathbb{R}$ the sets of rational numbers, positive integers, integers and real numbers, separately.
For  $\gamma\in \mathbb{Z}$ and a prime number $p$,
set $\mathbb{Q}_{p}^{n}$ be a vector space over the $p$-adic field $\mathbb{Q}_{p}$, $B_{\gamma} (x)$ represent a $p$-adic ball with center $x \in \mathbb{Q}_{p}^{n}$ and radius $p^{\gamma}$ (for the notations and notions, see  \cref{sec:preliminary} below).

Let  $0 \le \alpha<n$,   the $p$-adic  fractional maximal function of locally integrable function $f$ defined as
%----------------
\begin{align*}
%-------
  \mathcal{M}_{\alpha}^{p}(f)(x) = \sup_{\gamma\in \mathbb{Z}  \atop x\in \mathbb{Q}_{p}^{n}} \dfrac{1}{|B_{\gamma} (x)|_{h}^{1-\alpha/n}} \dint_{B_{\gamma} (x)} |f(y)| \mathd y,
%----------
\end{align*}
%-----------
where the supremum is taken over all $p$-adic balls $B_{\gamma} (x)\subset \mathbb{Q}_{p}^{n}$ and $|E|_{h}$ represents the Haar measure of a measurable set $E\subset\mathbb{Q}_{p}^{n}$.
When $\alpha=0$, we simply write  $\mathcal{M}^{p}$ instead of $\mathcal{M}_{0}^{p}$, which is the $p$-adic Hardy-Littlewood maximal function is defined by
%----------------
\begin{align*}   %\label{equ:p-adic-max-operator}
%-------
  \mathcal{M}^{p}(f)(x) = \sup_{\gamma\in \mathbb{Z}  \atop x\in \mathbb{Q}_{p}^{n}} \dfrac{1}{|B_{\gamma} (x)|_{h}} \dint_{B_{\gamma} (x)} |f(y)| \mathd y.
%----------
\end{align*}
%------------

The reader can refer to Stein \cite{stein1993harmonic} for the definition on the Euclidean case.

Similar to  \labelcref{equ:commutator-1},    two different kinds of commutator of the fractional maximal function are defined as follows (or see \cite{wu2023characterizationC}).
%--------------
\begin{definition} \label{def.commutator-frac-max}
%-----------------
 Let  $0 \le \alpha<n$ and $b$ be  a locally integrable function on $\mathbb{Q}_{p}^{n}$.
%----------
\begin{enumerate}[label=(\arabic*)]
%------------
\item The maximal commutator of $ \mathcal{M}_{\alpha}^{p}$ with $b$ is given by
%----------
\begin{align*}
%-------
\mathcal{M}_{\alpha,b}^{p} (f)(x) &= \sup_{\gamma\in \mathbb{Z}  \atop x\in \mathbb{Q}_{p}^{n}} \dfrac{1}{|B_{\gamma} (x)|_{h}^{1-\alpha/n}}  \dint_{B_{\gamma} (x)} |b(x)-b(y)| |f(y)| \mathd y,
%----------
\end{align*}
%------------
where the supremum is taken over all $p$-adic balls $B_{\gamma} (x)\subset \mathbb{Q}_{p}^{n}$.
%-----------
 \item  The nonlinear commutators generated by   $\mathcal{M}_{\alpha}^{p}$ and $b$  is defined by
%-------
\begin{align*}
%-------
[b,\mathcal{M}_{\alpha}^{p}] (f)(x) &= b(x) \mathcal{M}_{\alpha}^{p} (f)(x) -\mathcal{M}_{\alpha}^{p}(bf)(x).
%----------
\end{align*}
%-------
\end{enumerate}
%--------------
\end{definition}
%------------

When $\alpha=0$, we simply write $[b,\mathcal{M}^{p}]=[b,\mathcal{M}_{0}^{p}]$ and $\mathcal{M}_{b}^{p}=\mathcal{M}_{0,b}^{p}$.

We say that $[b,\mathcal{M}_{\alpha}^{p}] $  is a   nonlinear commutator due to the fact that it is not even a sublinear operator,
although the commutator $[b,T]$ is a linear one. It is worth pointing out that the nonlinear commutator $[b,\mathcal{M}_{\alpha}^{p}] $ and the maximal commutator $\mathcal{M}_{\alpha,b}^{p}$ essentially differ from each other.
For instance, $\mathcal{M}_{\alpha,b}^{p}$ is positive and sublinear, but $[b,\mathcal{M}_{\alpha}^{p}] $ is neither positive nor sublinear

Let $M$  and $M_{\alpha}$  stand for the classical Hardy–Littlewood maximal function and the   fractional maximal function in $\mathbb{R}^{n}$, respectively.
For the case of $b$ belongs to $\bmo(\mathbb{R}^{n})$, the nonlinear commutator $[b, M]$ and  $[b, M_{\alpha}]$ have been studied by many authors in the Euclidean spaces, such as, Bastero et al. \cite{bastero2000commutators} studied the necessary and sufficient conditions for   the boundedness of $[b,M]$ in $L^{q}(\mathbb{R}^{n})$ for $1<q<\infty$. Zhang and Wu obtained similar results for the fractional maximal function in \cite{zhang2009commutators} and extended the mentioned results to variable exponent Lebesgue spaces in \cite{zhang2014commutatorsfor,zhang2014commutators}.

On the other hand, in 2009, Kim \cite{kim2009carleson} studied $p$-adic BMO spaces. Recently, Chac\'{o}n-Cort\'{e}s and
Rafeiro \cite{chacon2021fractional} established the boundedness of the fractional maximal  in the $p$-adic variable exponent Lebesgue spaces. And He and Li \cite{he2023necessary} showed  the characterization of  $p$-adic BMO spaces in view  of the boundedness of commutators of maximal function $\mathcal{M}^{p}$ in the context of the $p$-adic Lebesgue spaces and Morrey spaces.

Inspired by the above literature,  the purpose of this paper is to   study the boundedness of the $p$-adic fractional maximal commutator $ \mathcal{M}_{\alpha,b}^{p}$ or the nonlinear commutator $[b,\mathcal{M}_{\alpha}^{p}] $ generated by $p$-adic fractional maximal function $ \mathcal{M}_{\alpha}^{p} $  over  $p$-adic  (variable)  Lebesgue spaces or   Morrey spaces, where the symbols $b$ belong to the $p$-adic BMO spaces, by which some new characterizations of the $p$-adic version of BMO spaces are given.

Let $\alpha\ge 0$, for a fixed $p$-adic ball $B_{*}$, the fractional maximal function with relation to $B_{*}$ of a locally integrable function $f$ is given by
%----------------
\begin{align*}
%-------
%  \mathcal{M}_{\alpha,B_{*}}^{p}(f)(x) = \sup_{\gamma\in \mathbb{Z}  \atop B_{\gamma} (x)\subseteq B_{*}} \dfrac{1}{|B_{\gamma} (x)|_{h}^{1-\alpha/n}} \dint_{B_{\gamma} (x)} |f(y)| \mathd y,
  \mathcal{M}_{\alpha,B_{*}}^{p}(f)(x) = \sup_{\gamma\in \mathbb{Z}  \atop B_{\gamma} (x)\subset B_{*}} \dfrac{1}{|B_{\gamma} (x)|_{h}^{1-\alpha/n}} \dint_{B_{\gamma} (x)} |f(y)| \mathd y,
%----------
\end{align*}
%-----------
where the supremum is taken over all the $p$-adic  ball $B_{\gamma} (x)$ with $B_{\gamma} (x)\subset B_{*}$ for a fixed $p$-adic  ball $B_{*}$. When $\alpha= 0$,  we simply write $\mathcal{M}_{B_{*}}^{p}$ instead of $\mathcal{M}_{0,B_{*}}^{p}$.

Our main results can be stated as follows.
%The first part of this paper   gives  some new characterization of the spaces $ \Lambda_{\beta}(\mathbb{Q}_{p}^{n})$ via the boundedness of  $\mathcal{M}_{\alpha,b}^{p}$  and $[b,\mathcal{M}_{\alpha}^{p}] $ in the $p$-adic field context  (see  \cref{sec:preliminary} for the corresponding notions).
 Firstly, we consider the boundedness of  nonlinear commutator $[b,\mathcal{M}_{\alpha}^{p}] $ in the context of $p$-adic  variable exponent Lebesgue spaces when the symbol belongs to a  $p$-adic   $ \bmo(\mathbb{Q}_{p}^{n})$ (see   \cref{sec:preliminary}  below). And  some new characterizations of BMO  via such commutators are given.
%gives a new characterizations of the  Lipschitz spaces of  $p$-adic version in the context of $p$-adic  variable exponent Lebesgue spaces.

%-----------------
\begin{theorem} \label{thm:nonlinear-frac-max-var-bmo}  %[Spanne-type  result]
%---------------
 Let    $0< \alpha<n$  and $b$ be a locally integrable function on $\mathbb{Q}_{p}^{n}$.
 Then the following assertions are equivalent:
%--------------
\begin{enumerate}[label=(T.\arabic*)]   %\arabic
%%--------------
\item   $b\in  \bmo(\mathbb{Q}_{p}^{n})$ and $b^{-}\in L^{\infty}(\mathbb{Q}_{p}^{n})$.
 %-----------
    \label{enumerate:thm-nonlinear-frac-max-var-bmo-1}
%------------------
   \item The commutator $ [b,\mathcal{M}_{\alpha}^{p}] $ is bounded from $L^{r(\cdot)}(\mathbb{Q}_{p}^{n})$ to $L^{q(\cdot)}(\mathbb{Q}_{p}^{n})$ for all $r(\cdot), q(\cdot)\in   \mathscr{C}^{\log}(\mathbb{Q}_{p}^{n}) $ with $r(\cdot)\in \mathscr{P}(\mathbb{Q}_{p}^{n})$, $ r_{+}<\frac{n}{\alpha}$ and $1/q(\cdot) = 1/r(\cdot) -\alpha/n$.
%-----------
    \label{enumerate:thm-nonlinear-frac-max-var-bmo-2}
%--------------------------------
\item  The commutator $ [b,\mathcal{M}_{\alpha}^{p}] $ is bounded from $L^{r(\cdot)}(\mathbb{Q}_{p}^{n})$ to $L^{q(\cdot)}(\mathbb{Q}_{p}^{n})$  for some  $r(\cdot), q(\cdot)\in   \mathscr{C}^{\log}(\mathbb{Q}_{p}^{n}) $ with $r(\cdot)\in \mathscr{P}(\mathbb{Q}_{p}^{n})$, $r_{+}<\frac{n}{\alpha}$ and $1/q(\cdot) = 1/r(\cdot) -\alpha/n$.
%-----------
    \label{enumerate:thm-nonlinear-frac-max-var-bmo-3}
%------------
   \item There exists some  $r(\cdot), q(\cdot)\in   \mathscr{C}^{\log}(\mathbb{Q}_{p}^{n}) $ with $r(\cdot)\in \mathscr{P}(\mathbb{Q}_{p}^{n})$,  $r_{+}<\frac{n}{\alpha}$ and $1/q(\cdot) = 1/r(\cdot) -\alpha/n$, such that
%-----------
\begin{align} \label{inequ:thm-nonlinear-frac-max-var-bmo-4}
%-----------
\sup_{\gamma\in \mathbb{Z} \atop x\in \mathbb{Q}_{p}^{n}}  \dfrac{\Big\| \big(b -|B_{\gamma} (x)|_{h}^{-\alpha/n}\mathcal{M}_{\alpha,B_{\gamma} (x)}^{p} (b) \big) \dchi_{B_{\gamma} (x)} \Big\|_{L^{q(\cdot)}(\mathbb{Q}_{p}^{n}) }}{\|\dchi_{B_{\gamma} (x)}\|_{L^{q(\cdot)}(\mathbb{Q}_{p}^{n}) }}  < \infty.
%-----------------
\end{align}
%-----------
%-----------
 \vspace{-1em}\label{enumerate:thm-nonlinear-frac-max-var-bmo-4}
%------------
   \item  For all   $r(\cdot), q(\cdot)\in   \mathscr{C}^{\log}(\mathbb{Q}_{p}^{n}) $ with $r(\cdot)\in \mathscr{P}(\mathbb{Q}_{p}^{n})$,  $r_{+}<\frac{n}{\alpha}$ and $1/q(\cdot) = 1/r(\cdot) -\alpha/n$, such that \labelcref{inequ:thm-nonlinear-frac-max-var-bmo-4} holds.
%-----------
\label{enumerate:thm-nonlinear-frac-max-var-bmo-5}
%-----------------
\end{enumerate}
%-------------
%----------
\end{theorem}
%-----------------

For the case of $r(\cdot)$ and $q(\cdot)$ being constants, we have the following result from \cref{thm:nonlinear-frac-max-var-bmo}, which is new even for this case.
%When  $r(\cdot)$ and $q(\cdot)$ are constants, we get the following result from \cref{thm:nonlinear-frac-max-var-bmo}.

%-----------------
\begin{corollary}  \label{cor:nonlinear-frac-max-bmo}
%---------------
 Let $0< \alpha< n$ and $b$ be a locally integrable function  on $\mathbb{Q}_{p}^{n}$. Then the following statements are equivalent:
%--------------
\begin{enumerate}[label=(C.\arabic*)]
%%--------------
\item   $b\in  \bmo(\mathbb{Q}_{p}^{n})$  and $b^{-}\in L^{\infty}(\mathbb{Q}_{p}^{n})$.
 %-----------
    \label{enumerate:cor-nonlinear-frac-max-bmo-1}
%------------------
   \item The commutator $[b,\mathcal{M}_{\alpha}^{p}]$ is bounded from $L^{r}(\mathbb{Q}_{p}^{n})$ to $L^{q}(\mathbb{Q}_{p}^{n})$ for all $r, q$ with $1<r<\frac{n}{\alpha}$ and $1/q = 1/r -\alpha/n$.
%-----------
    \label{enumerate:cor-nonlinear-frac-max-bmo-2}
%------------
   \item  The commutator $[b,\mathcal{M}_{\alpha}^{p}]$ is bounded from $L^{r}(\mathbb{Q}_{p}^{n})$ to $L^{q}(\mathbb{Q}_{p}^{n})$ for some $r, q$ with $1<r<\frac{n}{\alpha}$ and $1/q = 1/r -\alpha/n$.
%-----------
    \label{enumerate:cor-nonlinear-frac-max-bmo-3}
%------------
   \item  There exists some  $r, q$ with $1<r<\frac{n}{\alpha}$ and $1/q = 1/r -\alpha/n$, such that
%--------
\begin{align} \label{inequ:cor-nonlinear-frac-max-bmo-4}
%-----------
 \sup_{\gamma\in \mathbb{Z} \atop x\in \mathbb{Q}_{p}^{n}}  \dfrac{1}{|B_{\gamma} (x)|_{h}} \dint_{B_{\gamma} (x)} \Big|b(y)-|B_{\gamma} (x)|_{h}^{-\alpha/n} \mathcal{M}_{\alpha,B_{\gamma} (x)}^{p} (b)(y) \Big|^{q} \mathd y  <\infty.
%-----------------
\end{align}
%-----------
 \vspace{-1em}\label{enumerate:cor-nonlinear-frac-max-bmo-4}
%-----------------
   \item  For all   $r, q$ with  $1<r<\frac{n}{\alpha}$ and $1/q = 1/r -\alpha/n$, such that \labelcref{inequ:cor-nonlinear-frac-max-bmo-4} holds.
%--------
\label{enumerate:cor-nonlinear-frac-max-bmo-5}
%---------
\end{enumerate}
%-------------
%----------
\end{corollary}
%-----------------

%--------------------
\begin{remark}   \label{rem.cor-nonlinear-frac-max-bmo}
%-----------
\begin{enumerate}[ label=(\roman*)]  %,itemindent=-0.3em,itemindent=1.5em
%-------
\item For the case $\alpha=0$,
    \cref{cor:nonlinear-frac-max-bmo} was proved in \cite[Theorem 1.3]{he2023necessary}.
%-------
\item  Moreover, it is pointed out in \cite[Theorem 1.3]{he2023necessary} (see also \cref{lem:thm1.3-max-bmo-he2023necessary} below) that  $b\in  \bmo(\mathbb{Q}_{p}^{n})$  and $b^{-}\in L^{\infty}(\mathbb{Q}_{p}^{n})$ if and only if
%--------
\begin{align} \label{inequ:1.4-he2023necessary}
%-----------
 \sup_{\gamma\in \mathbb{Z} \atop x\in \mathbb{Q}_{p}^{n}} \dfrac{1}{|B_{\gamma} (x)|_{h}} \dint_{B_{\gamma} (x)} \Big|b(y)-  \mathcal{M}_{B_{\gamma} (x)}^{p} (b)(y) \Big|^{q} \mathd y  <\infty.
%-----------------
\end{align}
%-----------
Compared with \labelcref{inequ:1.4-he2023necessary},  \labelcref{inequ:cor-nonlinear-frac-max-bmo-4} gives a new characterization for $ \bmo(\mathbb{Q}_{p}^{n})$.
%-------
\end{enumerate}
%--------------------------
\end{remark}
%-------------

In particular, when $\alpha=0$, the results are also true  come from \cref{thm:nonlinear-frac-max-var-bmo} and \cref{cor:nonlinear-frac-max-bmo}.
Now we only give the case  in   the context of the $p$-adic version of  variable exponent  Lebesgue spaces, and it is new.

%-----------------
\begin{corollary}  \label{cor:nonlinear-max-var-bmo}
%---------------
 Let   $b$ be a locally integrable function  on $\mathbb{Q}_{p}^{n}$. Then the following statements are equivalent:
%--------------
\begin{enumerate}[label=(C.\arabic*)]
%%--------------
\item  $b\in  \bmo(\mathbb{Q}_{p}^{n})$  and $b^{-}\in L^{\infty}(\mathbb{Q}_{p}^{n})$.
 %-----------
    \label{enumerate:cor-nonlinear-max-var-bmo-1}
%------------------
   \item The commutator $ [b,\mathcal{M}^{p}] $ is bounded on  $L^{q(\cdot)}(\mathbb{Q}_{p}^{n})$ for all $  q(\cdot)\in   \mathscr{P}^{\log}(\mathbb{Q}_{p}^{n})  $.
%-----------
    \label{enumerate:cor-nonlinear-max-var-bmo-2}
%--------------------------------
\item  The commutator $ [b,\mathcal{M}^{p}] $ is bounded on $L^{q(\cdot)}(\mathbb{Q}_{p}^{n})$  for some  $  q(\cdot)\in   \mathscr{P}^{\log}(\mathbb{Q}_{p}^{n})  $.
%-----------
    \label{enumerate:cor-nonlinear-max-var-bmo-3}
%------------
   \item There exists some  $  q(\cdot)\in   \mathscr{P}^{\log}(\mathbb{Q}_{p}^{n})  $, such that
%-----------
\begin{align} \label{inequ:cor-nonlinear-max-var-bmo-4}
%-----------
\sup_{\gamma\in \mathbb{Z} \atop x\in \mathbb{Q}_{p}^{n}}  \dfrac{\Big\| \big(b - \mathcal{M}_{B_{\gamma} (x)}^{p} (b) \big) \dchi_{B_{\gamma} (x)} \Big\|_{L^{q(\cdot)}(\mathbb{Q}_{p}^{n}) }}{\|\dchi_{B_{\gamma} (x)}\|_{L^{q(\cdot)}(\mathbb{Q}_{p}^{n}) }}  < \infty.
%-----------------
\end{align}
%-----------
\vspace{-1em}\label{enumerate:cor-nonlinear-max-var-bmo-4}
%------------
   \item  For all   $  q(\cdot)\in   \mathscr{P}^{\log}(\mathbb{Q}_{p}^{n})  $, such that \labelcref{inequ:cor-nonlinear-max-var-bmo-4} holds.
%-----------
\label{enumerate:cor-nonlinear-max-var-bmo-5}
%---------
\end{enumerate}
%-------------
%----------
\end{corollary}
%-----------------

Next, we discuss the mapping properties of $\mathcal{M}_{\alpha,b}^{p}$  in the context of $p$-adic  variable exponent Lebesgue spaces when the symbol belongs to a  $p$-adic   $ \bmo(\mathbb{Q}_{p}^{n})$. And  some new characterizations of BMO  via such commutators are given.
%-----------------
\begin{theorem} \label{thm:frac-max-var-bmo}
%---------------
 Let   $0< \alpha <n$ and $b$ be a locally integrable function on $\mathbb{Q}_{p}^{n}$.
 Then the following assertions are equivalent:
%--------------
\begin{enumerate}[label=(T.\arabic*)]   %\arabic
%%--------------
\item   $b\in  \bmo(\mathbb{Q}_{p}^{n})$.
 %-----------
    \label{enumerate:thm-frac-max-var-bmo-1}
%------------------
   \item The commutator $ \mathcal{M}_{\alpha,b}^{p}$ is bounded from $L^{r(\cdot)}(\mathbb{Q}_{p}^{n})$ to $L^{q(\cdot)}(\mathbb{Q}_{p}^{n})$ for all $r(\cdot), q(\cdot)\in   \mathscr{C}^{\log}(\mathbb{Q}_{p}^{n}) $ with $r(\cdot)\in \mathscr{P}(\mathbb{Q}_{p}^{n})$, $ r_{+}<\frac{n}{\alpha }$ and $1/q(\cdot) = 1/r(\cdot) -\alpha/n$.
%-----------
    \label{enumerate:thm-frac-max-var-bmo-2}
%--------------------------------
\item  The commutator $ \mathcal{M}_{\alpha,b}^{p}$  is bounded from $L^{r(\cdot)}(\mathbb{Q}_{p}^{n})$ to $L^{q(\cdot)}(\mathbb{Q}_{p}^{n})$  for some  $r(\cdot), q(\cdot)\in   \mathscr{C}^{\log}(\mathbb{Q}_{p}^{n}) $ with $r(\cdot)\in \mathscr{P}(\mathbb{Q}_{p}^{n})$,  $ r_{+}<\frac{n}{\alpha }$ and $1/q(\cdot) = 1/r(\cdot) -\alpha/n$.
%-----------
    \label{enumerate:thm-frac-max-var-bmo-3}
%------------
   \item There exists some  $r(\cdot), q(\cdot)\in   \mathscr{C}^{\log}(\mathbb{Q}_{p}^{n}) $ with $r(\cdot)\in \mathscr{P}(\mathbb{Q}_{p}^{n})$, $ r_{+}<\frac{n}{\alpha }$ and $1/q(\cdot) = 1/r(\cdot) -\alpha/n$, such that
%-----------
\begin{align} \label{inequ:thm-frac-max-var-bmo-4}
%-----------
 \sup_{\gamma\in \mathbb{Z} \atop x\in \mathbb{Q}_{p}^{n}}  \dfrac{\Big\| \big(b -b_{B_{\gamma}(x)} \big) \dchi_{B_{\gamma} (x)} \Big\|_{L^{q(\cdot)}(\mathbb{Q}_{p}^{n}) }}{\|\dchi_{B_{\gamma} (x)}\|_{L^{q(\cdot)}(\mathbb{Q}_{p}^{n}) }}  < \infty.
%-----------------
\end{align}
%-----------
\vspace{-1em} \label{enumerate:thm-frac-max-var-bmo-4}
%------------
   \item  For all   $r(\cdot), q(\cdot)\in   \mathscr{C}^{\log}(\mathbb{Q}_{p}^{n}) $ with $r(\cdot)\in \mathscr{P}(\mathbb{Q}_{p}^{n})$,    $ r_{+}<\frac{n}{\alpha }$ and $1/q(\cdot) = 1/r(\cdot) -\alpha/n$, such that \labelcref{inequ:thm-frac-max-var-bmo-4} holds.
%-----------
\label{enumerate:thm-frac-max-var-bmo-5}
%-----------------
\end{enumerate}
%-------------
%----------
\end{theorem}
%-----------------

%For the case of $r(\cdot)$ and $q(\cdot)$ being constants, we have the following result from \cref{thm:frac-max-var-bmo}, which is new even for this case.
When  $r(\cdot)$ and $q(\cdot)$ are constants, we get the following result from \cref{thm:frac-max-var-bmo}.

%-----------------
\begin{corollary}  \label{cor:frac-max-bmo}
%---------------
 Let   $0< \alpha< n$ and $b$ be a locally integrable function on $\mathbb{Q}_{p}^{n}$. Then the following statements are equivalent:
%--------------
\begin{enumerate}[label=(C.\arabic*)]
%%--------------
\item   $b\in  \bmo(\mathbb{Q}_{p}^{n})$.
 %-----------
    \label{enumerate:cor-frac-max-bmo-1}
%------------------
   \item The commutator $\mathcal{M}_{\alpha,b}^{p}$ is bounded from $L^{r}(\mathbb{Q}_{p}^{n})$ to $L^{q}(\mathbb{Q}_{p}^{n})$ for all $r, q$ with $1<r<\frac{n}{\alpha}$ and $1/q = 1/r -\alpha/n$.
%-----------
    \label{enumerate:cor-frac-max-bmo-2}
%--------------------------------
\item  The commutator $\mathcal{M}_{\alpha,b}^{p}$  is bounded from $L^{r}(\mathbb{Q}_{p}^{n})$ to $L^{q}(\mathbb{Q}_{p}^{n})$ for some $r, q$ with $1<r<\frac{n}{\alpha}$ and $1/q = 1/r -\alpha/n$.
%-----------
    \label{enumerate:cor-frac-max-bmo-3}
%------------
   \item There exists some  $r, q$ with $1<r<\frac{n}{\alpha}$ and $1/q = 1/r -\alpha/n$, such that
%-----------
\begin{align} \label{inequ:cor-frac-max-bmo-4}
%-----------
  \sup_{\gamma\in \mathbb{Z} \atop x\in \mathbb{Q}_{p}^{n}}  \dfrac{1}{|B_{\gamma} (x)|_{h}} \dint_{B_{\gamma} (x)} \Big|b(y)-b_{B_{\gamma}(x)} \Big|^{q} \mathd y <\infty.
%-----------------
\end{align}
%-----------
 \vspace{-1em} \label{enumerate:cor-frac-max-bmo-4}
%------------
   \item  For all   ${r, q}$ with  $1<r<\frac{n}{\alpha}$ and $1/q = 1/r -\alpha/n$, such that \labelcref{inequ:cor-frac-max-bmo-4} holds.
%-----------
\label{enumerate:cor-frac-max-bmo-5}
%-----------------
\end{enumerate}
%-------------
%----------
\end{corollary}
%-----------------

%--------------------
\begin{remark}   \label{rem.cor-frac-max-bmo}
%-------
\begin{enumerate}[ label=(\roman*)]
%-------
\item   For the case $\alpha=0$,  \cref{cor:frac-max-bmo} is also holds, and the equivalence of \labelcref{enumerate:cor-frac-max-bmo-1}, \labelcref{enumerate:cor-frac-max-bmo-2} and \labelcref{enumerate:cor-frac-max-bmo-3}  was proved in \cite[Theorem 1.1]{he2023necessary}.
%-------
\item  Moreover,  the equivalence of \labelcref{enumerate:cor-frac-max-bmo-1}, \labelcref{enumerate:cor-frac-max-bmo-4} and \labelcref{enumerate:cor-frac-max-bmo-5}  is contained in \cref{lem:cor-5.17-kim2009carleson} below.
%-------
\end{enumerate}
%-------------
%-------
\end{remark}
%-------------

 For the case $\alpha=0$,  we give the follows result, which is valid and new,   from  \cref{thm:frac-max-var-bmo}.
%-----------------
\begin{corollary}  \label{cor:frac-max-var-bmo}
%---------------
 Let  $b$ be a locally integrable function  on $\mathbb{Q}_{p}^{n}$. Then the following statements are equivalent:
%--------------
\begin{enumerate}[label=(C.\arabic*)]
%%--------------
\item   $b\in  \bmo(\mathbb{Q}_{p}^{n})$.
 %-----------
    \label{enumerate:cor-frac-max-var-bmo-1}
%------------------
   \item The commutator $ \mathcal{M}_{b}^{p}$ is bounded on  $L^{q(\cdot)}(\mathbb{Q}_{p}^{n})$ for all $  q(\cdot)\in   \mathscr{P}^{\log}(\mathbb{Q}_{p}^{n}) $.
%-----------
    \label{enumerate:cor-frac-max-var-bmo-2}
%--------------------------------
\item  The commutator $ \mathcal{M}_{b}^{p}$  is bounded on $L^{q(\cdot)}(\mathbb{Q}_{p}^{n})$  for some  $  q(\cdot)\in   \mathscr{P}^{\log}(\mathbb{Q}_{p}^{n}) $.
%-----------
    \label{enumerate:cor-frac-max-var-bmo-3}
%------------
   \item There exists some  $  q(\cdot)\in   \mathscr{P}^{\log}(\mathbb{Q}_{p}^{n}) $  such that
%-----------
\begin{align} \label{inequ:cor-frac-max-var-bmo-4}
%-----------
 \sup_{\gamma\in \mathbb{Z} \atop x\in \mathbb{Q}_{p}^{n}}  \dfrac{\Big\| \big(b -b_{B_{\gamma}(x)} \big) \dchi_{B_{\gamma} (x)} \Big\|_{L^{q(\cdot)}(\mathbb{Q}_{p}^{n}) }}{\|\dchi_{B_{\gamma} (x)}\|_{L^{q(\cdot)}(\mathbb{Q}_{p}^{n}) }}  < \infty.
%-----------------
\end{align}
%-----------
\vspace{-1em}\label{enumerate:cor-frac-max-var-bmo-4}
%------------
   \item  For all   $  q(\cdot)\in   \mathscr{P}^{\log}(\mathbb{Q}_{p}^{n}) $  such that \labelcref{inequ:cor-frac-max-var-bmo-4} holds.
%-----------
\label{enumerate:cor-frac-max-var-bmo-5}
%---------
\end{enumerate}
%-------------
%----------
\end{corollary}
%-----------------

Finally we give  the mapping properties of $\mathcal{M}_{\alpha,b}^{p}$ and $[b,\mathcal{M}_{\alpha}^{p}] $ on $p$-adic  Morrey spaces when the symbol belongs to a  $p$-adic   $ \bmo(\mathbb{Q}_{p}^{n})$ (see   \cref{sec:preliminary}  below).

%-----------------------
\begin{theorem} \label{thm:nonlinear-frac-max-bmo-morrey}
%-----------------------
Suppose that   $b$ is a locally integrable function  on $\mathbb{Q}_{p}^{n}$,   $0 <\alpha <n$. $1 <r< n/\alpha$ and $0<\lambda <n-\alpha r$.
 % The following statements are equivalent:
%-----------------------
\begin{enumerate}[label=(T.\arabic*)]
 %%-------------------------------------
  \item  If  $1/q=1/r-\alpha/(n-\lambda)$.  Then $b\in  \bmo(\mathbb{Q}_{p}^{n})$  and $b^{-}\in L^{\infty}(\mathbb{Q}_{p}^{n})$  if and only if $ [b,\mathcal{M}_{\alpha}^{p}] $ is bounded from  $L^{r,\lambda}(\mathbb{Q}_{p}^{n})$ to $L^{q,\lambda}(\mathbb{Q}_{p}^{n})$.
%-----------
    \label{enumerate:thm-nonlinear-frac-max-bmo-morrey-1}
%--------------------------------
   \item If  $1/q=1/r-\alpha/n$ and $\lambda/r =\mu/q$.  Then $b\in  \bmo(\mathbb{Q}_{p}^{n})$  and $b^{-}\in L^{\infty}(\mathbb{Q}_{p}^{n})$  if and only if   $ [b,\mathcal{M}_{\alpha}^{p}] $ is bounded from  $L^{r,\lambda}(\mathbb{Q}_{p}^{n})$ to $L^{q,\mu}(\mathbb{Q}_{p}^{n})$.
%-----------
    \label{enumerate:thm-nonlinear-frac-max-bmo-morrey-2}
%------------------------
\end{enumerate}
%--------------------
\end{theorem}
%-----------------

%-----------------------
\begin{theorem} \label{thm:frac-max-bmo-morrey}
%-----------------------
 Assume that   $b$ be a locally integrable function  on $\mathbb{Q}_{p}^{n}$,   $0 <\alpha <n$, $1 <r< n/\alpha$ and $0<\lambda <n-\alpha r$.
%-----------------------
\begin{enumerate}[label=(T.\arabic*)]
 %%-------------------------------------
  \item   If  $1/q=1/r-\alpha/(n-\lambda)$.  Then $b\in  \bmo(\mathbb{Q}_{p}^{n})$  if and only if $\mathcal{M}_{\alpha,b}^{p}$  is bounded from  $L^{r,\lambda}(\mathbb{Q}_{p}^{n})$ to $L^{q,\lambda}(\mathbb{Q}_{p}^{n})$.
%-----------
    \label{enumerate:thm-frac-max-bmo-morrey-1}
%--------------------------------
   \item If  $1/q=1/r-\alpha/n$ and $\lambda/r =\mu/q$.  Then $b\in  \bmo(\mathbb{Q}_{p}^{n})$   if and only if $\mathcal{M}_{\alpha,b}^{p}$  is bounded from  $L^{r,\lambda}(\mathbb{Q}_{p}^{n})$ to $L^{q,\mu}(\mathbb{Q}_{p}^{n})$.
%-----------
    \label{enumerate:thm-frac-max-bmo-morrey-2}
%--------------------------------
\end{enumerate}
%--------------------
\end{theorem}
%------------

%--------------------
\begin{remark}   \label{rem.frac-max-bmo-morrey}
%-------
\begin{enumerate}[ label=(\roman*)]
%-------
\item   For the case $\alpha=0$,  \cref{thm:nonlinear-frac-max-bmo-morrey} and \cref{thm:frac-max-bmo-morrey}  are also true.
%-------
\item    \cref{thm:nonlinear-frac-max-bmo-morrey} \labelcref{enumerate:thm-nonlinear-frac-max-bmo-morrey-2} was proved in \cite[Theorem 1.6]{he2023necessary} (for $\alpha=0$), and \cref{thm:frac-max-bmo-morrey} \labelcref{enumerate:thm-frac-max-bmo-morrey-2} was proved in \cite[Theorem 1.5]{he2023necessary} (for $\alpha=0$).
%-------
\end{enumerate}
%-------------
%-------
\end{remark}
%-------------

 Throughout this paper, the letter $C$  always stands for a constant  independent of the main parameters involved and whose value may differ from line to line.
In addition, here and hereafter $|E|_{h}$  will always denote the Haar measure of a measurable set $E$ on $\mathbb{Q}_{p}^{n}$ and by  \raisebox{2pt}{$\dchi_{E}$} denotes the  characteristic function of a measurable set $E\subset\mathbb{Q}_{p}^{n}$.
%Let $L^{p} ~(1\le p\le \infty)$  be the standard $L^{p} $-space with respect to the Haar measure $\mathd x$.
%For a measurable set $E \subset\mathbb{Q}_{p}^{n}$ and a positive integer $m$, we will use the notation $(E)^{m}=\underbrace{E\times \cdots \times E}_{m}$ sometimes. And we will occasionally use the notational $\vec{f}=(f_{1},\dots , f_{m})$, $T(\vec{f})=T(f_{1},\dots , f_{m})$, $\mathd\vec{y}=dy_{1}\cdots  dy_{m}$ and $(x,\vec{y})=(x,y_{1},\dots , y_{m})$ for convenience.

%This paper is organized as follows. In the next section, we recall some basic definitions and known results. In  \cref{sec:proof-mab}, we will prove  \cref{thm:lipschitz-frac-main-1}.   \cref{sec:proof-nonlinear}  is devoted to proving \cref{thm:lipschitz-nonlinear-frac-main-1}.

\section{Preliminaries and lemmas}
\label{sec:preliminary}

In order to prove the  principal  results of this paper, we firstly review some necessary  concepts and remarks.

\subsection{$p$-adic field $\mathbb{Q}_{p}$}

%Firstly, we introduce some basic and necessary notations for the $p$-adic field.

Set $p \ge 2$ be a fixed prime number in $\mathbb{Z}$ and $G_{p}=\{0,1,\ldots,p-1\}$.
For every non-zero rational number $x$, by the unique factorization theorem, there is a unique $\gamma=\gamma(x)\in \mathbb{Z}$, such that $x=p^{\gamma} \frac{m}{n}$, where $m,n\in \mathbb{Z}$  are not divisible by $p$ (i.e. $p$ is coprime to $m$, $n$).
Define the mapping $|\cdot|_{p}: \mathbb{Q} \to \mathbb{R}^{+}$  as follows:
%-------------------
\begin{align*} %\label{equ:1.1-hussain2021boundedness} %\cite{hussain2021boundedness}
%-------------
|x|_{p}=
%-------------
\begin{cases}
%-------------
p^{-\gamma} & \text{if} \ x \neq 0,
%-------------
\\
%-------------
0 & \text{if} \ x = 0.
%-------------
\end{cases}
%-------------
\end{align*}
%-------------
 The $p$-adic absolute value $|\cdot|_{p}$ is endowed with many properties of the usual real norm $|\cdot|$ with an
additional non-Archimedean property (i.e., $\{|m|_{p}, m\in \mathbb{Z}\}$ is bounded)
$$|x+y|_{p} \le \max \{ |x|_{p},|y|_{p}\}.$$
In addition,  $|\cdot|_{p}$  also satisfies the following properties:
%-----------------------
\begin{enumerate}[label=(\arabic*)]
%%-------------------------------------
\item  (positive definiteness) $|x|_{p} \ge 0$. Specially,  $|x|_{p}=0 \Leftrightarrow  ~x = 0$
%--------------------
\item (multiplicativity) $|xy|_{p} = |x|_{p} |y|_{p} $.
%--------------------
\item (non-Archimedean triangle inequality)  $|x+y|_{p} \le \max\{ |x|_{p}, |y|_{p}\} $. The equality holds if and only if $|x|_{p} \neq |y|_{p}$.
% Specially, when $|x|_{p} \neq |y|_{p}$,   $|x+y|_{p} =\max\{ |x|_{p}, |y|_{p}\} $.
%-------------------------------
\end{enumerate}
%-------------
 Denote by $\mathbb{Q}_{p}$ the $p$-adic field which is defined as the completion of the field of rational numbers $\mathbb{Q}$ with respect to the   $p$-adic absolute value $|\cdot|_{p}$.

From the standard $p$-adic analysis, any non-zero element $x\in\mathbb{Q}_{p}$ can be  uniquely represented as a canonical series form
%-----------------
\begin{align*}
%-------
 x=p^{\gamma}(a_{0}+a_{1}p+a_{2}p^{2}+\cdots)  =p^{\gamma}   \sum_{j=0}^{\infty} a_{j}p^{j},
%----------
\end{align*}
%------------
where $a_{j}\in G_{p}$ and $a_{0}\neq 0$, and  $\gamma =\gamma(x)\in \mathbb{Z}$ is called as the $p$-adic valuation of $x$. The series converges in the $p$-adic  absolute value since the inequality $| a_{j}p^{j}|_{p} \le  p^{-j}$  holds for all   $j\in \mathbb{N}$.
%The $p$-adic absolute value ensures the convergence of series \labelcref{equ:1.2-hussain2021boundedness} in $\mathbb{Q}_{p}$, because the inequality $|p^{\gamma} \beta_{i}p^{i} |_{p} \le p^{-\gamma-i} $ holds for all $\gamma \in \mathbb{Z}$ and $i \in \mathbb{N}$

Moreover, the $n$-dimensional  $p$-adic  vector space $\mathbb{Q}_{p}^{n} =\mathbb{Q}_{p}\times\cdots\times \mathbb{Q}_{p}~(n\ge 1)$, consists of all points $x= (x_{1},\ldots,x_{n})$, where $x_{i} \in \mathbb{Q}_{p} ~(i=1,\ldots,n)$,  equipped with the following absolute value
%-----------------
\begin{align*}
%-------
 |x|_{p}= \max_{1\le j\le n} |x_{j}|_{p}.
%----------
\end{align*}
%------------
For $\gamma\in \mathbb{Z}$ and $a= (a_{1},a_{2},\dots,a_{n})\in \mathbb{Q}_{p}^{n}$, we denote by
%-------------
\begin{align*}
%------------------
  B_{\gamma} (a)&=    \{ x\in \mathbb{Q}_{p}^{n}: |x-a|_{p}\le p^{\gamma}\}
  % \colorbox{red!20}{$=\bigcup_{k\le \gamma} S_{k} (a)$????}
%--------------------
\end{align*}
%-------------
the closed ball with the center at $a$ and radius $p^{\gamma}$ and by
%-------------
\begin{align*}
%------------------
  S_{\gamma} (a)&=    \{ x\in  \mathbb{Q}_{p}^{n}: |x-a|_{p}= p^{\gamma}\}  =B_{\gamma}(a)\setminus B_{\gamma-1}(a)
  %\colorbox{green!20}{$=B_{\gamma}(a)\setminus B_{\gamma-1}(a)$????}
%--------------------
\end{align*}
%-------------
the corresponding sphere.
For $a=0$, we write $B_{\gamma} (0) = B_{\gamma} $, and $S_{\gamma} (0) = S_{\gamma} $.
Note that $B_{\gamma} (a) =\bigcup\limits_{k\le \gamma} S_{k} (a)$ % \cite{mo2019padic,mo2018commutator}
and $\mathbb{Q}_{p}^{n}\setminus \{0\}= \bigcup\limits_{\gamma\in \mathbb{Z}} S_{\gamma} $.
It is easy to see that the equalities
%-------------
\begin{align*}
%------------------
 a_{0} +B_{\gamma} = B_{\gamma} (a_{0})
 \ \text{and} \
 a_{0} +S_{\gamma} = S_{\gamma} (a_{0}) = B_{\gamma} (a_{0}) \setminus B_{\gamma-1} (a_{0})
%--------------------
\end{align*}
%-------------
hold for all $a_{0} \in \mathbb{Q}_{p}^{n}$ and $\gamma\in \mathbb{Z}$.

%It follows from non-Archimedean triangle inequality that two balls $B_{\gamma} (x)$ and $B_{\gamma'} (y)$ either do not intersect or one of
%them is contained in the other, which differ from those of the Euclidean case. And note that in the second case under conditions $\gamma=\gamma'$ these balls are equal.
 The following properties can   be found in \cite{kim2009q} (see Lemma 3.1), it shows that two balls $B_{\gamma} (x)$ and $B_{\gamma'} (y)$ either do not intersect or one of them is contained in the other, which differ from those of the Euclidean case.
%--------------
\begin{lemma}\label{lem:lem-3.1-kim2009q}
%-------
 Let $\gamma, \gamma'\in \mathbb{Z}$,  $x,y\in  \mathbb{Q}_{p}^{n}$.  The family  $\mathcal{B}_{p}=\{B_{\gamma} (x): \gamma, \in \mathbb{Z}, x \in  \mathbb{Q}_{p}^{n}\}$   consisting of all   $p$-adic balls has the following properties:
%-------
\begin{enumerate}[label=(\arabic*)] %fullwidth,
%------------
\item If   $\gamma\le\gamma'$, then  either $B_{\gamma} (x) \cap B_{\gamma'} (y)=\emptyset$ or $B_{\gamma} (x) \subset B_{\gamma'} (y)$
%-----------
\item $B_{\gamma} (x)= B_{\gamma} (y)$   if and only if  $y\in B_{\gamma} (x)$.
%----------
\end{enumerate}
%------------
\end{lemma}
%------------

Since $\mathbb{Q}_{p}^{n}$ is a locally compact commutative group with respect to addition, there exists a unique
Haar measure $\mathd x$  on $\mathbb{Q}_{p}^{n}$ (up to positive constant multiple) which is translation invariant (i.e., $\mathd (x+a )= \mathd x$), such that
%-------------
\begin{align*}
%------------------
  \dint_{B_{0}} \mathd x &=  |B_{0}|_{h} =1,
%--------------------
\end{align*}
%-------------
where $|E|_{h}$ denotes the Haar measure of measurable subset $E $ of $\mathbb{Q}_{p}^{n}$. Furthermore, from this integral theory, it is
easy to obtain that %a simple calculation shows that
%-------------
\begin{align*}    % \cite{taibleson1975fourier} P.8
%--------------
  \dint_{B_{\gamma}(a)} \mathd x &=|B_{\gamma}(a)|_{h}      = p^{n\gamma}
%------------
\\ \intertext{and}
%------------
 \dint_{S_{\gamma}(a)} \mathd x &= |S_{\gamma}(a)|_{h}     =  p^{n\gamma}(1-p^{-n})= |B_{\gamma}(a)|_{h} -|B_{\gamma-1}(a)|_{h}      \notag
%--------------------
\end{align*}
%-------------
hold for all $a \in \mathbb{Q}_{p}^{n}$ and $\gamma\in \mathbb{Z}$.

For more information about the p-adic field, we refer readers to \cite{taibleson1975fourier,vladimirov1994padic}.

%------------------------
\subsection{$p$-adic variable Lebesgue spaces}  %function spaces
%------------------------

In what follows, we say that a  real-valued  measurable
function $f$ defined on $\mathbb{Q}_{p}^{n}$ is in $L^{q}(\mathbb{Q}_{p}^{n})$, $1\le q\le \infty$, if it satisfies
%-----------------
\begin{align} \label{equ:4-he2022characterization}
%-------
\|f\|_{L^{q}(\mathbb{Q}_{p}^{n})} &=  \Big(\dint_{\mathbb{Q}_{p}^{n}}  |f(x)|^{q} \mathd x \Big)^{1/q} <\infty,\ \ 1\le q<\infty
%------------
\end{align}
%--------
and denote by $L^{\infty}(\mathbb{Q}_{p}^{n})$   the set of all measurable real-valued functions  $f$ on $\mathbb{Q}_{p}^{n}$ satisfying
%--------
\begin{align*}
%------------
  \|f\|_{L^{\infty}(\mathbb{Q}_{p}^{n})}  &=\esssup_{x\in \mathbb{Q}_{p}^{n}} |f(x)|
%  = \inf\Big\{ \alpha: \big|\{x\in\mathbb{Q}_{p}^{n}:~|  f(x)| >\alpha \}\big|_{h} =0\Big\}
<\infty.
%------------
\end{align*}
%------------
Here, %when $q=1$,
the integral in equation \labelcref{equ:4-he2022characterization} is defined as follows:
%-----------------
\begin{align*} % \label{equ:1.3-kim2009carleson}
\begin{aligned}
%-------
   \dint_{\mathbb{Q}_{p}^{n}}  |f(x)|^{q} \mathd x  &=\lim_{\gamma\to\infty} \dint_{B_{\gamma}(0)}  |f(x)|^{q} \mathd x
%------------
   =\lim_{\gamma\to\infty} \sum_{-\infty<k\le \gamma} \dint_{S_{k}(0)}  |f(x)|^{q} \mathd x,
%------------
\end{aligned}
\end{align*}
%------------
if the limit exists.

 Some often used computational principles are worth noting. In particular,  if $f\in  L^{1}(\mathbb{Q}_{p}^{n}) $, then
%------------
\begin{align*}
%------------
   \dint_{\mathbb{Q}_{p}^{n}}  f(x)  \mathd x  &= \sum_{\gamma=-\infty}^{+\infty} \dint_{S_{\gamma}}  f(x)  \mathd x
%------------
\\ \intertext{and}
%------------
 \dint_{\mathbb{Q}_{p}^{n}} f(tx)  \mathd x  &= \frac{1}{|t|_{p}^{n}} \dint_{\mathbb{Q}_{p}^{n}} f(x) \mathd x,
%------------
\end{align*}
%------------
where $t\in \mathbb{Q}_{p}\setminus\{0\}$, $tx= (tx_{1},\ldots,tx_{n})$ and $\mathd (tx)=|t|_{p}^{n} \mathd x$.
%change of variables

We now introduce the notion of $p$-adic variable exponent Lebesgue spaces and give some properties needed in
the sequel (see \cite{chacon2020variable} for the respective proofs).

We say that a measurable function $q(\cdot)$ is a variable exponent if $q(\cdot): \mathbb{Q}_{p}^{n}\to (0,\infty)$.

%--------------
\begin{definition} \label{def.variable-exponent}
%-----------------
  Given a measurable function $q(\cdot)$ defined on $\mathbb{Q}_{p}^{n}$,   we denote by
$$
q_{-} :=\essinf_{x\in \mathbb{Q}_{p}^{n}} q(x),\ \
q_{+}:= \esssup_{x\in \mathbb{Q}_{p}^{n}} q(x).$$
%----------
\begin{enumerate}[label=(\arabic*)] %fullwidth,
%------------ see P14 \cite{cruz2013variable}
\item $q'_{-}=\essinf\limits_{x\in \mathbb{Q}_{p}^{n}} q'(x)=\frac{q_{+}}{q_{+}-1},\ \ q'_{+}= \esssup\limits_{x\in \mathbb{Q}_{p}^{n}} q'(x)=\frac{q_{-}}{q_{-}-1}.$
%----------
% \item  Denote by $\mathscr{P}_{0}(\mathbb{Q}_{p}^{n})$ the set of all measurable functions $ q(\cdot): \mathbb{Q}_{p}^{n}\to(0,\infty)$ such that
%$$0< q_{-}\le q(x) \le q_{+}<\infty,\ \ x\in \mathbb{Q}_{p}^{n}.$$
%%------------
%\item  Denote by $\mathscr{P}_{1}(\mathbb{Q}_{p}^{n})$ the set of all measurable functions $ q(\cdotp): \mathbb{Q}_{p}^{n}\to[1,\infty)$ such that
%$$1\le q_{-}\le q(x) \le q_{+}<\infty,\ \ x\in \mathbb{Q}_{p}^{n}.$$
%-----------
  \item Denote by $\mathscr{P}(\mathbb{Q}_{p}^{n})$ the set of all measurable functions $ q(\cdot): \mathbb{Q}_{p}^{n}\to(1,\infty)$ such that
$$1< q_{-}\le q(x) \le q_{+}<\infty,\ \ x\in \mathbb{Q}_{p}^{n}.$$
%----------
 \item  The set $\mathscr{B}(\mathbb{Q}_{p}^{n})$ consists of all  measurable functions  $q(\cdot)\in\mathscr{P}(\mathbb{Q}_{p}^{n})$ satisfying that the Hardy-Littlewood maximal operator $\mathcal{M}^{p}$ is bounded on $L^{q(\cdot)}(\mathbb{Q}_{p}^{n})$.
%-------
\end{enumerate}
%--------------
\end{definition}
%------------

%--------------
\begin{definition}[$p$-adic variable exponent Lebesgue spaces] \label{def.p-adic-lebesgue-space}
%-----------------
 Let   $q(\cdot) \in \mathscr{P}(\mathbb{Q}_{p}^{n})$.
%----------
Define the $p$-adic variable exponent Lebesgue spaces $L^{q(\cdot)}(\mathbb{Q}_{p}^{n})$ as follows
%----------------
\begin{align*}
%-------
  L^{q(\cdot)}(\mathbb{Q}_{p}^{n})=\{f~ \text{is measurable function}: \int_{\mathbb{Q}_{p}^{n}} \Big( \frac{|f(x)|}{\eta} \Big)^{q(x)} \mathrm{d}x<\infty ~\text{for some constant}~ \eta>0\}.
%----------
\end{align*}
%------------
 The Lebesgue space $L^{q(\cdot)}(\mathbb{Q}_{p}^{n})$ is a Banach function space with respect to the Luxemburg norm
 \begin{equation*}
   \|f\|_{L^{q(\cdot)}(\mathbb{Q}_{p}^{n})}=\inf \Big\{ \eta>0:   \int_{\mathbb{Q}_{p}^{n}} \Big( \frac{|f(x)|}{\eta} \Big)^{q(x)} \mathrm{d}x \le 1 \Big\}.
\end{equation*}
%----------
\end{definition}
%------------

%--------------
\begin{definition}[$\log$-H\"{o}lder continuity\cite{chacon2020variable}] \label{def.4.1-4.4-log-holder}
%-----------------
 Let measurable function $q(\cdot) \in \mathscr{P}(\mathbb{Q}_{p}^{n})$.
%----------
\begin{enumerate}[label=(\arabic*)] %fullwidth,
%------------
\item Denote by  $\mathscr{C}_{0}^{\log}(\mathbb{Q}_{p}^{n})$ the set of all %local  $\log$-H\"{o}lder continuous  functions
    $q(\cdotp)$ which satisfies
 \begin{equation*}    %\label{equ.2.1}
  \gamma\Big(  q_{-}(B_{\gamma}(x)) -  q_{+}(B_{\gamma}(x)) \Big)  \le C
\end{equation*}
for all $\gamma\in \mathbb{Z}$ and any $x\in\mathbb{Q}_{p}^{n}$, where $C$ denotes a universal positive constant.
%----------
 \item  The set $\mathscr{C}_{\infty}^{\log}(\mathbb{Q}_{p}^{n})$ consists of all % $\log$-H\"{o}lder continuous functions
     $ q(\cdot)$ which  satisfies
\begin{equation*}%\label{equ.2.3}
 |q(x)-q(y)| \le \frac{C}{\log_{p}(p+\min\{|x|_{p},|y|_{p}\})}
\end{equation*}
%------------
 for any $x, y\in\mathbb{Q}_{p}^{n}$, where $C$ is a  positive constant.
%-----------
 \item Denote by $\mathscr{C}^{\log}(\mathbb{Q}_{p}^{n}) =\mathscr{C}_{0}^{\log}(\mathbb{Q}_{p}^{n}) \bigcap \mathscr{C}_{\infty}^{\log}(\mathbb{Q}_{p}^{n})$ the set of all global log-H\"{o}lder continuous functions $q(\cdotp)$.
%-------
% \item    Denote by $\mathscr{P}^{\log}(\mathbb{Q}_{p}^{n}) =\mathscr{C}(\mathbb{Q}_{p}^{n}) \bigcap \mathscr{P}(\mathbb{Q}_{p}^{n})$.
%-------
\end{enumerate}
%--------------
\end{definition}
%------------

In what follows, we denote   $ \mathscr{C}^{\log}(\mathbb{Q}_{p}^{n}) \bigcap \mathscr{P}(\mathbb{Q}_{p}^{n})$ by  $\mathscr{P}^{\log}(\mathbb{Q}_{p}^{n})  $.
And  for a function $b$ defined on $\mathbb{Q}_{p}^{n}$, we denote
%-----------
\begin{align*}
%-------
  b^{-}(x) :=- \min\{b, 0\} =
%-------
\begin{cases}
%------------
 0,  & \text{if}\ b(x) \ge 0  \\
 |b(x)|, & \text{if}\ b(x) < 0
%------------
\end{cases}
%----------
\end{align*}
%------------
and  $b^{+}(x) =|b(x)|-b^{-}(x)$. Obviously, $b(x)=b^{+}(x)-b^{-}(x)$.

%The first part of \cref{lem.thm-5.2-variable-max-bounded} may be found in \cite{chacon2020variable} (see Theorem 5.2).  By elementary calculations, the second of \cref{lem.thm-5.2-variable-max-bounded} can  be obtained as well.
The following properties are  useful  (or see \cite[Lemma 2.2]{wu2023characterizationC}).

\begin{lemma}  \label{lem.thm-5.2-variable-max-bounded}
Let $q(\cdot)\in \mathscr{P}(\mathbb{Q}_{p}^{n} )$.
%----------
\begin{enumerate}[label=(\arabic*)] %fullwidth,
%------------
\item  (\cite[Theorem 5.2]{chacon2020variable}) If $q(\cdot)\in \mathscr{C}^{\log}(\mathbb{Q}_{p}^{n})$,  then  $ q(\cdot)\in \mathscr{B}(\mathbb{Q}_{p}^{n})$.
%----------
 \item  The  following conditions are equivalent: %\color{red}
%------------
  \begin{enumerate}[label=(\roman*),align=left]
%-----------
  \item  $ q(\cdot)\in \mathscr{B}(\mathbb{Q}_{p}^{n})$,
 \item   $q'(\cdot)\in \mathscr{B}(\mathbb{Q}_{p}^{n})$,
  \item    $ q(\cdot)/q_{0}\in \mathscr{B}(\mathbb{Q}_{p}^{n})$ for some $1<q_{0}<q_{-}$,
 \item    $ (q(\cdot)/q_{0})'\in \mathscr{B}(\mathbb{Q}_{p}^{n})$ for some $1<q_{0}<q_{-}$,
%-----------
\end{enumerate}
%----------
where $r'$ stand for the conjugate exponent of $r$, viz., $1=\frac{1}{r(\cdot)} + \frac{1}{r'(\cdot)}$.
%----------
\end{enumerate}
%--------------
\end{lemma}
%-------------

%--------------------
\begin{remark}   \label{rem.variable-max-bounded}
%-------------
If  $ q(\cdot)\in \mathscr{B}(\mathbb{Q}_{p}^{n})$ and $s>1$, then $ s q(\cdot)\in \mathscr{B}(\mathbb{Q}_{p}^{n})$ (for the Euclidean case see Remark 2.13 of \cite{cruz2006theboundedness} for more details).
%Similar results of Euclidean space can be referred to Remark 2.13 of \cite{cruz2006theboundedness} for further details.
%------------
\end{remark}
%----------

 The following results regarding the H\"{o}lder's inequality are required (or see \cite[Lemma 2.3]{wu2023characterizationC}).
%The part \labelcref{enumerate:holder-p-adic}   is known as the  H\"{o}lder's inequality on Lebesgue spaces over  $p$-adic vector space $\mathbb{Q}_{p}^{n}$.
%And similar to the Euclidean case, the part \labelcref{enumerate:holder-p-adic-variable} can be  deduced  by simple calculations (or see Lemma 3.8 in \cite{chacon2020variable}).  % straightforward to show
%--------------------------------
\begin{lemma}[Generalized H\"{o}lder's inequality on  $\mathbb{Q}_{p}^{n}$] \label{lem:holder-inequality-p-adic}
 Let $\mathbb{Q}_{p}^{n}$ be an $n$-dimensional $p$-adic vector space.
 %For any $\gamma \in \mathbb{Z}$,  $x \in  \mathbb{Q}_{p}^{n}$, the $p$-adic ball   $B_{\gamma}(x)\subset  \mathbb{Q}_{p}^{n}$.
%--------------------------
\begin{enumerate}[label=(\roman*)] %\arabicfullwidth,
\setlength{\itemsep}{0pt}
%------------
\item Suppose that  $1\le q \le\infty$ with $\frac{1}{q}+\frac{1}{q'}=1$,   and measurable functions $f\in L^{q}(\mathbb{Q}_{p}^{n})$ and $g\in L^{q'}(\mathbb{Q}_{p}^{n})$.  Then there exists a positive constant $C$ such that
%---------------
\begin{align*}
%-------
   \dint_{\mathbb{Q}_{p}^{n}} |f(x)g(x)|  \mathrm{d}x \le C \|f\|_{L^{q}(\mathbb{Q}_{p}^{n})} \|g\|_{L^{q'}(\mathbb{Q}_{p}^{n})}.
%----------
\end{align*}
%------------
 \vspace{-1em} \label{enumerate:holder-p-adic}
%----------------
\item  Suppose that     $ q_{1}(\cdot), q_{2}(\cdot), r(\cdot) \in \mathscr{P}(\mathbb{Q}_{p}^{n})$ and   $r(\cdot) $ satisfy $\frac{1}{r(\cdot) }=\frac{1}{q_{1}(\cdot)}+ \frac{1}{q_{2}(\cdot)}$ almost everywhere.
Then there exists a positive constant $C$ such that the inequality
%----------------
\begin{align*}
%-------
  \|fg\|_{L^{r(\cdot)}(\mathbb{Q}_{p}^{n})}\le C \|f \|_{L^{q_{1}(\cdot)}(\mathbb{Q}_{p}^{n})}  \|g \|_{L^{q_{2}(\cdot)}(\mathbb{Q}_{p}^{n})}
%----------
\end{align*}
%------------
holds for all  $f \in L^{q_{1}(\cdot)}(\mathbb{Q}_{p}^{n})$ and  $g \in L^{q_{2}(\cdot)}(\mathbb{Q}_{p}^{n})$.
%-----------
  \label{enumerate:holder-p-adic-variable}
%----------
\item  When $r(\cdot)= 1$ in  \labelcref{enumerate:holder-p-adic-variable}  as mentioned above, we have $ q_{1}(\cdot), q_{2}(\cdot)  \in \mathscr{P}(\mathbb{Q}_{p}^{n})$ and   $\frac{1}{q_{1}(\cdot)}+ \frac{1}{q_{2}(\cdot)}=1$ almost everywhere.
Then there exists a positive constant $C$ such that the inequality
%----------------
\begin{align*}
%-------
    \dint_{\mathbb{Q}_{p}^{n}} |f(x)g(x)|  \mathrm{d}x  \le C \|f \|_{L^{q_{1}(\cdot)}(\mathbb{Q}_{p}^{n})}  \|g \|_{L^{q_{2}(\cdot)}(\mathbb{Q}_{p}^{n})}
%----------
\end{align*}
%------------
holds for all  $f \in L^{q_{1}(\cdot)}(\mathbb{Q}_{p}^{n})$ and  $g \in L^{q_{2}(\cdot)}(\mathbb{Q}_{p}^{n})$.
%-----------
    \label{enumerate:holder-p-adic-variable-1}
%----------
\end{enumerate}
%--------------
\end{lemma}
%-------------

%By elementary calculations we have the following property. It can also  be found in \cite{guliyev2022some}, when  Young function $\Phi(t)=t^{p}$.
The following results for the characteristic function $\dchi_{B_{\gamma}(x)}$ are  required as well (see \cite[Lemma 2.4]{wu2023characterizationC}).
%By elementary calculations, the first part may be obtain  from  the $p$-adic integral theory (or refer to  \labelcref{equ:p-adic-integral-theory}).% Lemma 11 when $\lambda=0$ in \cite{he2022characterization}.
%The second part may be founded in \cite{chacon2021fractional} (see Lemma 7), and the part (4) follows from  \labelcref{enumerate:holder-p-adic-variable} of  \cref{lem:holder-inequality-p-adic}. Moreover, according to \cref{lem.thm-5.2-variable-max-bounded} and \labelcref{enumerate:charact-norm-p-adic-variable-chacon2021fractional} of \cref{lem:norm-characteristic-p-adic}, the third part  can also be deduced by simple calculations. So, we omit the proofs.
%--------------------------------
\begin{lemma}[Norms of characteristic functions]\label{lem:norm-characteristic-p-adic}
%--------------------------
Let $\mathbb{Q}_{p}^{n}$ be an $n$-dimensional $p$-adic vector space.
%-------
 \begin{enumerate}[label=(\arabic*)] %fullwidth,
%------------
\item If $1\le q<\infty$. Then there exist a constant $C > 0 $ such that
%---------------
\begin{align*}
%-------
  \|\dchi_{B_{\gamma}(x)}\|_{L^{q}(\mathbb{Q}_{p}^{n})}    &= |B_{\gamma}(x)|_{h}^{1/q}=p^{n\gamma/q}.
%----------
\end{align*}
%------------
 \vspace{-1em} \label{enumerate:charact-norm-p-adic-he}
%----------------
\item  (\cite[Lemma 7]{chacon2021fractional}) If $q(\cdot)\in  \mathscr{C}^{\log}(\mathbb{Q}_{p}^{n})$. Then
%---------------
\begin{align*}
%-------
  \|\dchi_{B_{\gamma}(x)}\|_{L^{q(\cdot)}(\mathbb{Q}_{p}^{n})}    & \le C p^{n\gamma/q(x,\gamma)},
%----------
\end{align*}
%------------
where
%-------------
\begin{align*}
%-------
q(x,\gamma)=
%-------------
\begin{cases}
%-------------
 q(x) & \text{if} \ \gamma < 0,
%-------------
\\
%-------------
q(\infty) & \text{if} \ \gamma \ge 0.
%-------------
\end{cases}
%----------
%----------
\end{align*}
%------------
  \vspace{-1em}\label{enumerate:charact-norm-p-adic-variable-chacon2021fractional}
%-------
\item  If $q(\cdot)\in  \mathscr{C}^{\log}(\mathbb{Q}_{p}^{n})$ and $q(\cdot)\in  \mathscr{P}(\mathbb{Q}_{p}^{n} )$. Then there exist a constant $C > 0 $ such that
%---------------
\begin{align*}
%-------
  \dfrac{1}{|B_{\gamma} (x)|_{h} }  \big\|\dchi_{B_{\gamma} (x)} \big\|_{L^{q(\cdot)}(\mathbb{Q}_{p}^{n}) }  \big\|   \dchi_{B_{\gamma} (x)} \big\|_{L^{q'(\cdot)}(\mathbb{Q}_{p}^{n}) } <C
%----------
\end{align*}
%------------
holds for all  $p$-adic ball $B_{\gamma} (x)\subset \mathbb{Q}_{p}^{n}$.
%------------
 \label{enumerate:charact-norm-conjugate-p-adic-variable}
%-------
\item Let $0 <\alpha<n$. If  $ q(\cdot) $, $r(\cdot)\in \mathscr{P}(\mathbb{Q}_{p}^{n})$  with   $r_{+}<\frac{n}{\alpha}$ and $1/q(\cdot) = 1/r(\cdot) - \alpha/n$, then  there exists a  constant $C>0$ such that
%----------------
\begin{align*}
%-------
  \|\dchi_{B_{\gamma} (x)} \|_{L^{r(\cdot)}(\mathbb{Q}_{p}^{n})}\le C  |B_{\gamma} (x)|_{h}^{\alpha/n} \|\dchi_{B_{\gamma} (x)}  \|_{L^{q(\cdot)}(\mathbb{Q}_{p}^{n})}
%----------
\end{align*}
%------------
holds for all  $p$-adic balls $B_{\gamma} (x)\subset \mathbb{Q}_{p}^{n}$.
%-----------
    \label{enumerate:charact-norm-fraction-p-adic-variable}
%----------
\end{enumerate}
%------------
\end{lemma}
%-------------

%-------------
\subsection{$p$-adic BMO spaces and Morrey spaces}  %function spaces
%------------

The  bounded mean oscillation function  space $\bmo $ was introduced by John and Nirenberg \cite{john1961functions} which plays a crucial role in harmonic analysis.

%-----------------
\begin{definition}[$p$-adic BMO space]  \label{def.5.3-bmo-space-kim2009carleson}
%---------------
%Let  $f\in L_{\loc}^1 (\mathbb{Q}_{p}^n)$.  $\bmo(\mathbb{Q}_{p}^{n})$  is the set of all locally integrable functions $f$ on $\mathbb{Q}_{p}^n$  satisfy
A  locally integrable function $f$  is said to belong to $\bmo(\mathbb{Q}_{p}^n)$ if the seminorm given by
%-------
\begin{align*}
%------------------
  \|f\|_{p,*} &=\|f\|_{\bmo(\mathbb{Q}_{p}^n)} = \sup_{\gamma\in \mathbb{Z} \atop x\in \mathbb{Q}_{p}^{n}} \dfrac{1}{|B_{\gamma} (x)|_{h}} \dint_{B_{\gamma} (x)} \big|f(y)-f_{B_{\gamma} (x)}\big| \mathd y  %<\infty,
%--------------------
\end{align*}
%-----------
is finite, where  the supremum is taken over all balls $B_{\gamma}(x)$ on $\mathbb{Q}_{p}^{n}$, and  $f_{B_{\gamma} (x)}$ denotes the average of $f$ over $B_{\gamma} (x)$, i.e.,
$f_{B_{\gamma}(x)} = |B_{\gamma}(x)|_{h}^{-1} \int_{B_{\gamma}(x)}   f(y)  \mathd y$.
%------------
\end{definition}
%--------------

 The following properties can be found in \cite[Theorem 5.16 and Corollary 5.17]{kim2009carleson}.
%--------------
\begin{lemma}\label{lem:cor-5.17-kim2009carleson}
%-------
 Let   $f\in \bmo(\mathbb{Q}_{p}^{n})$, the following properties hold:
%-------
\begin{enumerate}[label=(\alph*)] %fullwidth,
%------------
\item  If $0<q<\infty$, then the norm $ \|f\|_{\bmo_{q}(\mathbb{Q}_{p}^{n})} \le C_{pq} \|f\|_{\bmo(\mathbb{Q}_{p}^{n})}$   for some positive constant $C_{pq}$, where
$\|f\|_{\bmo_{q} (\mathbb{Q}_{p}^{n})}$   defined as
%-------
\begin{align*}
%------------------
 \|f\|_{p,*,q}= \|f\|_{\bmo_{q}(\mathbb{Q}_{p}^n)} = \sup_{\gamma\in \mathbb{Z} \atop x\in \mathbb{Q}_{p}^{n}} \Big(\dfrac{1}{|B_{\gamma} (x)|_{h}} \dint_{B_{\gamma} (x)} \big|f(y)-f_{B_{\gamma} (x)}\big|^{q} \mathd y\Big)^{1/q}.
%------------
\end{align*}
%-----------
  \vspace{-1em} \label{enumerate:cor-5.17-a-kim2009carleson}
%-------
\item If  $q>1$, then the norm $\|f\|_{\bmo(\mathbb{Q}_{p}^{n})} $ is equivalent to the norm $\|f\|_{\bmo_{q} (\mathbb{Q}_{p}^{n})}$.
%-----------
 \label{enumerate:cor-5.17-b-kim2009carleson}
%-------
\item \textbf(John–Nirenberg type inequality)\ There exist some constants $c_{1}, c_{2} > 0$ depending only on $p$ and $n$ such that
%-------
\begin{align} \label{inequ:John-Nirenberg-ine}
%---------
  \big|\{x\in B: |f(x)-f_{B}|>t\}\big|_{h} \le c_{1} \exp\Big(-\frac{c_{2}t}{\|f\|_{p,*}}\Big) |B|_{h}, \qquad t>0
%------------
\end{align}
%-----------
for any $p$-adic ball $B\in \mathcal{B}_{p}$.
%-----------
 \label{enumerate:thm-5.16-kim2009carleson}
%-------
\item For any $\lambda$ with $0<\lambda<c_{2}/\|f\|_{p,*}$, where  $c_{2}$ is the constant  given in  \eqref{inequ:John-Nirenberg-ine},
%-------
\begin{align*}
%------------------
    \sup_{\gamma\in \mathbb{Z} \atop x\in \mathbb{Q}_{p}^{n}} \dfrac{1}{|B_{\gamma} (x)|_{h}} \dint_{B_{\gamma} (x)} \exp\big(\lambda|f(y)-f_{B_{\gamma} (x)}|\big)  \mathd y<\infty.
%------------
\end{align*}
%-------
 \label{enumerate:cor-5.17-c-kim2009carleson}
%-------
\end{enumerate}
%------------
\end{lemma}
%------------

%For a function $f\in L_{\loc}^1 (\mathbb{Q}_{p}^n)$, we define the $p$-adic sharp maximal function $ \mathcal{M}_{p}^\sharp $ by
%%----------------
%\begin{align}  \label{equ:p-adic-sharp-operator}
%%-------
%  \mathcal{M}_{p}^{\sharp}(f)(x) = \sup_{\gamma\in \mathbb{Z} \atop x\in \mathbb{Q}_{p}^{n}} \dfrac{1}{|B_{\gamma} (x)|_{h}} \dint_{B_{\gamma} (x)} |f(y)-f_{B_{\gamma} (x)}| \mathd y.
%%----------
%\end{align}
%%------------
%  where   $f_{B_{\gamma} (x)}$ denotes the average of $f$ over $B_{\gamma} (x)$, i.e., $f_{B_{\gamma}(x)} = |B_{\gamma}(x)|_{h}^{-1} \int_{B_{\gamma}(x)}   f(y)  \mathd y$.

%%-------------
%\subsection{$p$-adic Morrey spaces}  %function spaces
%%------------

Morrey spaces, named after Morrey, were first introduced by Morrey \cite{morrey1938solutions} to study the local behavior of solutions to elliptic partial differential equations, namely, regularity problems arising in the calculus of variations.

The following gives the definition of {$p$-adic Morrey spaces.

%-----------------
\begin{definition}[$p$-adic Morrey space]  \label{def.2.3-morrey-p-adic-ma2020weighted}  %\cite{ma2020weighted}
%---------------
Let $1\le q<\infty$ and $\lambda\ge 0$, the $p$-adic Morrey spaces $  L^{q,\lambda}  (\mathbb{Q}_{p}^n)$ is defined as
%-------
\begin{align*}
%-----------
 L^{q,\lambda}(\mathbb{Q}_{p}^n) &=\{f\in  L_{\loc}^{q}  (\mathbb{Q}_{p}^{n}): \|f\|_{L^{q,\lambda}(\mathbb{Q}_{p}^n)} <\infty\},
%------------
\\ \intertext{and}
%---------
  \|f\|_{L^{q,\lambda}(\mathbb{Q}_{p}^{n})} &=  \sup_{\gamma\in \mathbb{Z} \atop x\in \mathbb{Q}_{p}^{n}} \Big(\dfrac{1}{|B_{\gamma} (x)|_{h}^{ \lambda/n}} \dint_{B_{\gamma} (x)}  |f(y) |^{q} \mathd y \Big)^{1/q},
%--------------------
\end{align*}
%-----------
  where  the supremum is taken over all balls $B_{\gamma}(x)$ on $\mathbb{Q}_{p}^{n}$.
%------------
\end{definition}
%--------------
In particular,
%-----------------
\begin{align*}
%-------
  L^{q,\lambda}(\mathbb{Q}_{p}^{n})  =
%-------
\begin{cases}
%------------
   L^{q}(\mathbb{Q}_{p}^{n})  & \text{if}\ \lambda=0, \\
  L^{\infty}(\mathbb{Q}_{p}^{n})  & \text{if}\ \lambda=n,\\
  \Theta             & \text{if}\ \lambda<0 \ \text{or}\ \lambda>n,
%------------
\end{cases}
%----------
\end{align*}
%------------
where $\Theta$ is the set of all functions equivalent to $0$ on $\mathbb{Q}_{p}^{n}$.

%-------------
\subsection{Some facts of the Orlicz spaces}
%------------

%We need some basic facts from the theory of Orlicz spaces. For more information about Orlicz spaces, we refer to \cite{rao1991theory}
The following gives some basic facts about Orlicz spaces that we need (refer to \cite{rao1991theory} for more information).

We call $\Psi$ a Young function defined on $[0,\infty)$, if $\Psi$ is a (left) continuous,  convex and increasing function with
%-----------
\begin{align*}
%-------
\begin{cases}
%------------
 \lim\limits_{t\to +0}\Psi(t) =\Psi(0) =0,    \\
 \lim\limits_{t\to\infty}\Psi(t) =\infty.
%------------
\end{cases}
%----------
\end{align*}
%--------

The Orlicz space $L^{\Psi}(\mathbb{Q}_{p}^{n})$ is  the set of all measurable functions $f$ such that
%-----------
\begin{align*}
%-------
 \dint_{\mathbb{Q}_{p}^{n}} \Psi \Big(\frac{|f(y)|}{\nu}\Big) \mathd y <\infty
%----------
\end{align*}
%--------
 for some $\nu>0$. It is a Banach space when endowed with the Luxemburg norm
%-----------
\begin{align*}
%-------
 \|f\|_{\Psi} =  \|f\|_{L^{\Psi}(\mathbb{Q}_{p}^{n})} = \inf\Big\{\nu>0: \dint_{\mathbb{Q}_{p}^{n}} \Psi \Big(\frac{|f(y)|}{\nu}\Big) \mathd y \le 1\Big\}.
%----------
\end{align*}
%--------
Specifically,
%-------------
\begin{align*}
%-------
L^{\Psi}(\mathbb{Q}_{p}^{n}) =
%------------
 \begin{cases}
 %----------
L^{r}(\mathbb{Q}_{p}^{n})  & \text{if}\ \Psi(t) = t^{r}, r\in [1,\infty),   \\
L^{\infty}(\mathbb{Q}_{p}^{n})  & \text{if}\ \Psi(t) = 0, t\in [0,1],     \\
L^{\infty}(\mathbb{Q}_{p}^{n})  & \text{if}\  \Psi(t) = \infty, t>1.
%------------
\end{cases}
%----------
\end{align*}
%---------

In addition, the $\Psi$-average of a function $f$ on a $p$-adic ball $B_{\gamma}(x)\subset \mathbb{Q}_{p}^{n}$  is defined as
%-----------
\begin{align*}
%-------
 \|f\|_{\Psi,B_{\gamma}(x)} =   \inf\Big\{\nu>0: \frac{1}{|B_{\gamma}(x)|_{h}}\dint_{B_{\gamma}(x)} \Psi \Big(\frac{|f(y)|}{\nu}\Big) \mathd y \le 1\Big\}.
%----------
\end{align*}
%--------
The following gives the generalized Jensen's inequality (see \cite[(2.10)]{lerner2009new} or \cite[Lemma 4.2]{lu2014multilinear}).
%----------------
\begin{lemma} \label{lem:4.2-lu2014multilinear}
%-----------------
 If  $\Psi_{1}$ and $\Psi_{2}$ are two Young functions with $\Psi_{1}\le \Psi_{2}$, for $t\ge t_{0}>0$, then there exists a positive constant $C$  such that    $\|f\|_{\Psi_{1},B_{\gamma}(x)}\le C \|f\|_{\Psi_{2},B_{\gamma}(x)}$.
%-------------
\end{lemma}
%-------------

Define the complementary  function  $\bar{\Psi}$ of $\Psi$   as
%-----------
\begin{align} \label{equ:complementary-orlicz-function}
%-------
 \bar{\Psi}(s)=
%-------
\begin{cases}
%------------
 \sup\limits_{0\le t<\infty} \{st-\Phi(t)\}   & s\in [0,\infty),  \\
  \infty  & s=\infty.
%------------
\end{cases}
%---------
\end{align}
%--------
Thus,  $\bar{\Psi}$ is also a Young function, and its $\bar{\Psi}$-average  is related to the $\Psi$-average via the generalized H\"{o}lder's inequality over  $p$-adic ball $B_{\gamma}(x)$ (refer to \cite{rao1991theory} or \cite[2.12]{lerner2009new}), that is,
%-----------
\begin{align} \label{equ:generalized-holder-orlicz}
%-------
   \frac{1}{|B_{\gamma}(x)|_{h}}\dint_{B_{\gamma}(x)}  |f(y)g(y)|  \mathd y \le  C\|f\|_{\Psi,B_{\gamma}(x)} \|g\|_{\bar{\Psi},B_{\gamma}(x)}
%----------
\end{align}
%--------

An interesting particular case is the Young function   $\Phi(t)=t(1+\log^{+}t)$ and $\Psi(t)=e^{t}-1$, which define  the classical Zygmund spaces $L(\log L)$ and $\exp L$, respectively.
The corresponding averages over a $p$-adic ball $B_{\gamma}(x)$ will be denoted by
%-----------
\begin{align*}
%-------
 \|\cdot\|_{\Phi,B_{\gamma}(x)} =   \|\cdot\|_{L(\log L),B_{\gamma}(x)}  \ \text{and}\ \|\cdot\|_{\Psi,B_{\gamma}(x)} =   \|\cdot\|_{\exp L,B_{\gamma}(x)}.
%----------
\end{align*}
%--------
A calculation reveals that the complementary function of $\Psi$ given by  \eqref{equ:complementary-orlicz-function} such that $\bar{\Psi}(t)\le \Phi(t)$.
It follows from \cref{lem:4.2-lu2014multilinear} and the generalized H\"{o}lder's inequality \eqref{equ:generalized-holder-orlicz} that
%-----------
\begin{align} \label{equ:generalized-holder-inequality-orlicz}
%-------
   \frac{1}{|B_{\gamma}(x)|_{h}}\dint_{B_{\gamma}(x)}  |f(y)g(y)|  \mathd y \le  C\|f\|_{\exp L,B_{\gamma}(x)} \|g\|_{L(\log L),B_{\gamma}(x)}.
%----------
\end{align}
%--------
For any $p$-adic ball $B_{\gamma}(x)\subset \mathbb{Q}_{p}^{n}$  and any $\bmo(\mathbb{Q}_{p}^n)$ function $b$, \cref{lem:cor-5.17-kim2009carleson} \labelcref{enumerate:cor-5.17-c-kim2009carleson} implies that for appropriate constant $C>0$
%-------
\begin{align*}
%------------------
  \|b-b_{B_{\gamma} (x)}\|_{\exp L,B_{\gamma}(x)} &\le C\|b\|_{\bmo(\mathbb{Q}_{p}^n)}.
%------------
\end{align*}
%-----------
This, along with \eqref{equ:generalized-holder-inequality-orlicz},  yields a generalized  H\"{o}lder's inequality which will be used in this article:
%-----------
\begin{align} \label{equ:generalized-holder-inequality-orlicz-bmo}
%-------
   \frac{1}{|B_{\gamma}(x)|_{h}}\dint_{B_{\gamma}(x)} \big| b(y)-b_{B_{\gamma} (x)}\big| |f(y)|  \mathd y \le  C\|b\|_{\bmo(\mathbb{Q}_{p}^n)} \|f\|_{L(\log L),B_{\gamma}(x)}.
%----------
\end{align}
%--------

%-------------
\subsection{Maximal functions and  pointwise estimates}
%------------
 %\cite{lu2014multilinear} \cite{lerner2009new}  \cite{guliyev2022some}

%---------------
\begin{definition} [$p$-adic (fractional) maximal functions] \label{def.p-adic-max-frac-orlicz} \
 Let   $f\in L_{\loc}^{1}(\mathbb{Q}_{p}^{n})$.
%-------------
\begin{enumerate}[label=(\alph*)]
%--------------
\item  If $\epsilon>0$, define the operator $\mathcal{M}_{\epsilon}^{p}(f) (x)= \big[ \mathcal{M}^{p}(|f|^{\epsilon})(x) \big]^{1/\epsilon}=\Big( \sup\limits_{\gamma\in \mathbb{Z}  \atop x\in \mathbb{Q}_{p}^{n}} \dfrac{1}{|B_{\gamma} (x)|_{h}} \dint_{B_{\gamma} (x)} |f(y)|^{\epsilon} \mathd y \Big)^{1/\epsilon}$.
%-------------
\item  If  $0\le\alpha<n$, the  (fractional) maximal functions related to Young function $\Phi(t)=t(1+\log^{+}t)$ are defined by
% define the maximal function related to Young function $\Phi(t)=t(1+\log^{+}t)$ via
\begin{align*}
\mathcal{M}_{ L(\log L)}^{p}(f)(x)   &= \sup\limits_{\gamma\in \mathbb{Z}  \atop x\in \mathbb{Q}_{p}^{n}}    \|f\|_{L(\log L),B_{\gamma}(x)}
%------------
\\ \intertext{and}
%------------
\mathcal{M}_{\alpha, L(\log L)}^{p}(f)(x)   &= \sup\limits_{\gamma\in \mathbb{Z}  \atop x\in \mathbb{Q}_{p}^{n}}  |B_{\gamma}(x)|_{h}^{\alpha/n}  \|f\|_{L(\log L),B_{\gamma}(x)},
\end{align*}
%----------
\end{enumerate}
where the supremum is taken over all $p$-adic balls $B_{\gamma} (x)\subset \mathbb{Q}_{p}^{n}$, and the Luxemburg type average $\|\cdot\|_{L(\log L),B}$ is defined as
$$ \|f\|_{L(\log L),B} = \inf \Big\{\nu>0:  \frac{1}{|B|_{h}} \int_{B}  \frac{|f(y)|}{\nu} \Big(1+\log^{+}\Big(\frac{|f(y)|}{\nu}\Big)  \mathrm{d}y \le 1 \Big\}$$
%$$ \|f\|_{L(\log L),B} = \inf \Big\{\nu>0:  \frac{1}{|B|_{h}} \int_{B}  \frac{|f(y)|}{\nu} \log\Big(e+\frac{|f(y)|}{\nu}\Big)  \mathrm{d}y \le 1 \Big\}.$$
with $\log^{+}(t)=\max\{\log(t),0\}$.
%--------
\end{definition}
%----------

%--------------
\begin{remark}  \label{rem.def-p-adic-max-frac-orlicz}
%-------------
\begin{enumerate}[label=(\alph*)]
%%------------
\item  If  we take $f\equiv 1$ in \eqref{equ:generalized-holder-inequality-orlicz}, it follows that for every  $\alpha\in [0,n)$ the
inequality
\begin{align*}
\mathcal{M}_{\alpha}^{p}(f)(x)  \le C \mathcal{M}_{\alpha,L(\log L)}^{p} (f)(x).
\end{align*}
%-----------
 \vspace{-2em}\label{enumerate:def-p-adic-max-frac-orlicz-1}
%----------
\item  Note from  \cite{bernardis2010composition}  that   $\mathcal{M}_{\alpha}^{p}(\mathcal{M}^{p}) \approx \mathcal{M}_{\alpha,L(\log L)}^{p}$  (refer also to \cite[Lemma 4.1]{bernardis2006weighted} or  \cite{he2023necessary}).  %, \cite[Lemma 1.6]{perez1995endpoint}
%-----------
\label{enumeratedef-p-adic-max-frac-orlicz-2}
%%-------------
\end{enumerate}
%-----------------
\end{remark}
%-----------

The following result, which is called the $p$-adic version of Lebesgue differentitation theorem, can be found in \cite{kim2009carleson} (see Corollary 2.11, or Theorem 1.14 in \cite{taibleson1975fourier}, P.29).
%----------------
\begin{lemma}[$p$-adic Lebesgue differentitation theorem]  \label{lem:cor2.11-kim2009carleson}
%------------
\begin{enumerate}[label=(\alph*)]  %,itemindent=1.5em
%-------------
\item  If $f\in L^{q}(\mathbb{Q}_{p}^{n})$ for $1\le q<\infty$, then we have that
%-------------
\begin{align*}
%------------------
  \lim_{\gamma\to -\infty} \Big\| \dfrac{1}{|B_{\gamma}(\cdot)|_{h}} \dint_{B_{\gamma}(\cdot)} f(y)\mathd y -f  \Big\|_{L^{q}(\mathbb{Q}_{p}^{n})} &= 0.
%--------------------
\end{align*}
%-------------
\item  If $f\in L_{\loc}^{1}(\mathbb{Q}_{p}^{n})$ is given, then we have that
%-------------
\begin{align*}
%------------------
 \Big| \Big\{ x\in \mathbb{Q}_{p}^{n}: \lim_{\gamma\to -\infty} \Big| \dfrac{1}{|B_{\gamma}(x)|_{h}} \dint_{B_{\gamma}(x)} f(y)\mathd y -f(x)  \Big|\neq 0 \Big\} \Big|_{h} &= 0.
%--------------------
\end{align*}
%-------------
\item  Especially, if $f\in L^{1}(\mathbb{Q}_{p}^{n})$, then, for $\almostevery x\in \mathbb{Q}_{p}^{n}$, we have that
%---------
\begin{align*}
%------------------
  \lim_{\gamma\to -\infty}   \dfrac{1}{|B_{\gamma}(x)|_{h}} \dint_{B_{\gamma}(x)} f(y)\mathd y =f(x).
%--------------------
\end{align*}
%-------------
\end{enumerate}
%-------------
%-------------
\end{lemma}
%------------

%%------------------------
%\subsection{Auxiliary propositions and lemmas}
%%------------------------

From the proof of Theorem 1.3 in \cite{he2023necessary}, we can obtain the following characterization of BMO functions.
%-----------------
\begin{lemma}  \label{lem:thm1.3-max-bmo-he2023necessary}
%---------------
 Let   $b$ be a locally integrable function on $\mathbb{Q}_{p}^{n}$. Then the following assertions are equivalent:
%--------------
\begin{enumerate}[label=(\arabic*)]
%--------------
\item   $b\in  \bmo(\mathbb{Q}_{p}^{n})$  and  $b^{-}\in L^{\infty}(\mathbb{Q}_{p}^{n})$.
 %-----------
     \label{enumerate:thm1.3-max-bmo-he2023necessary-1}
%-------
   \item For all $1\le q<\infty$,  there exists a positive constant $C$ such that
%--------
\begin{align} \label{inequ:max-bmo-he2023necessary}
%-----------
 \sup_{\gamma\in \mathbb{Z} \atop x\in \mathbb{Q}_{p}^{n}}  \dfrac{1}{|B_{\gamma} (x)|_{h}} \dint_{B_{\gamma} (x)} \Big|b(y)- \mathcal{M}_{B_{\gamma} (x)}^{p} (b)(y) \Big|^{q} \mathd y  \le C.
%-----------------
\end{align}
%--------
    \label{enumerate:thm1.3-max-bmo-he2023necessary-2}
%-----------
   \item \labelcref{inequ:max-bmo-he2023necessary} holds for some $1\le q<\infty$.
%-----------
     \label{enumerate:thm1.3-max-bmo-he2023necessary-3}
%-------
\end{enumerate}
%----------
\end{lemma}
%----------

The following gives the Hardy-Littlewood-Sobolev-type inequality for the  fractional maximal function $\mathcal{M}_{\alpha}^{p}$ on $p$-adic vector space
(or see \cite[Lemma 2.8]{wu2023characterizationC}).
%The first two parts of \cref{lem:frac-max-p-adic-estimate}  can be founded in \cite{he2022characterization} (see Lemma 8),  the third part comes from  \cite[Theorem 4]{chacon2021fractional},
% and the last part comes from  \cite[Theorem 5.2]{chacon2020variable}.
%%The following result comes from  Theorem 4 of \cite{chacon2021fractional}. % The following result follows from  Theorem 4 of \cite{chacon2021fractional}. % in the framework of variable exponent p-adic Lebesgue spaces.
%-----------------
\begin{lemma}  \label{lem:frac-max-p-adic-estimate}
%---------------
   Let $\mathbb{Q}_{p}^{n}$ be an $n$-dimensional $p$-adic vector space.
%--------------
\begin{enumerate}[label=(\roman*)]
%--------------
\item (\cite[Lemma 8]{he2022characterization}) If  $0 <\alpha<n$,  $1<r< n/\alpha$ such that $1/q = 1/r  -\alpha/n$, then  $\mathcal{M}_{\alpha}^{p} $ is bounded from $L^{r}(\mathbb{Q}_{p}^{n})$ to $L^{q}(\mathbb{Q}_{p}^{n})$.
%-----------
   \label{enumerate:lem-8-he2022characterization-1} %\cite{he2022characterization}
%-----------
\item (\cite[Lemma 8]{he2022characterization})  If  $0 <\alpha<n$, $ r=1$ such that $ q = n/(n -\alpha)$, then  $\mathcal{M}_{\alpha}^{p} $ is bounded from $L^{1}(\mathbb{Q}_{p}^{n})$ to $L^{q,\infty}(\mathbb{Q}_{p}^{n})$ (weak-type space).
%-----------
   \label{enumerate:lem-8-he2022characterization-2} %\cite{he2022characterization}
%------
 \item (\cite[Theorem 4]{chacon2021fractional})  If  $0 <\alpha<n$, $r(\cdot), q(\cdot)\in   \mathscr{C}^{\log}(\mathbb{Q}_{p}^{n}) $ with $r(\cdot)\in \mathscr{P}(\mathbb{Q}_{p}^{n})$,  $r_{+}<n/\alpha$ and $1/q(\cdot) = 1/r(\cdot) -\alpha/n$, then $\mathcal{M}_{\alpha}^{p} $ is bounded from $L^{r(\cdot)}(\mathbb{Q}_{p}^{n})$ to $L^{q(\cdot)}(\mathbb{Q}_{p}^{n})$.
%------------
    \label{enumerate:thm-4-chacon2021fractional} % \cite{chacon2021fractional}
%-----------
 \item  (\cite[Theorem 5.2]{chacon2020variable}) If  $q(\cdot)\in   \mathscr{P}^{\log}(\mathbb{Q}_{p}^{n})  $, then $\mathcal{M}^{p} $ is bounded from $L^{q(\cdot)}(\mathbb{Q}_{p}^{n})$ to itself (or see \cref{lem.thm-5.2-variable-max-bounded}).
%------------
    \label{enumerate:thm-5.2-chacon2020variable} % \cite{chacon2020variable}
%-----------
\end{enumerate}
%----------
%----------
\end{lemma}
%----------

%----------------
\begin{lemma}[{\cite[Themorem 1]{kim2009q}}]  \label{lem:thm1.1-kim2009q}
%------------
If  $1<q<\infty$,   then   $\mathcal{M}^{p}$ is bounded from $L^{q}(\mathbb{Q}_{p}^{n}) $ to itself.
%--------
\end{lemma}
%----------

By  \cref{lem:frac-max-p-adic-estimate} \labelcref{enumerate:lem-8-he2022characterization-1}, if  $0 <\alpha<n$,  $1<r< n/\alpha$ and $f\in L^{r}(\mathbb{Q}_{p}^{n})$, then $\mathcal{M}_{\alpha}^{p}(f)(x)<\infty $ almost everywhere. A similar result is also valid in variable Lebesgue spaces (see \cite[Lemma 2.9]{wu2023characterizationC}).
%-----------------
\begin{lemma}  \label{lem:frac-max-almost-every}
%---------------
   Let $0 <\alpha<n$, $r(\cdot)\in \mathscr{P}(\mathbb{Q}_{p}^{n})$ and $1<r_{-}\le r_{+}<n/\alpha$. If   $f\in L^{r(\cdot)}(\mathbb{Q}_{p}^{n})$, then  $\mathcal{M}_{\alpha}^{p}(f)(x)<\infty $ for almost everywhere $x\in \mathbb{Q}_{p}^{n}$.
%----------
\end{lemma}
%----------

The following results are required in the subsequent proof of the main results.
%----------------
\begin{lemma}[\cite{he2022characterization}]  \label{lem:lem10-he2022characterization}
%------------
Let $0<\alpha<n$, $1<r<n/\alpha$ and $0<\lambda<n-r\alpha$
\begin{enumerate}[label=(\alph*)]  %,itemindent=1.5em
%-------------
\item  If $1/q=1/r-\alpha/(n-\lambda)$, then   $\mathcal{M}_{\alpha}^{p}$ is bounded from $L^{r,\lambda}(\mathbb{Q}_{p}^{n}) $ to $L^{q,\lambda}(\mathbb{Q}_{p}^{n}) $.
%--------
\item  If $1/q=1/r-\alpha/n$ and $\lambda/r =\mu/q$, then   $\mathcal{M}_{\alpha}^{p}$ is bounded from $L^{r,\lambda}(\mathbb{Q}_{p}^{n}) $ to $L^{q,\mu}(\mathbb{Q}_{p}^{n}) $.
%-------------
\end{enumerate}
%-------------
%-------------
\end{lemma}
%------------

 Noting that $\mathcal{M}^{p}$ is   bounded on $L^{q}(\mathbb{Q}_{p}^{n}) $  from \cref{lem:thm1.1-kim2009q}, together with \cref{def.2.3-morrey-p-adic-ma2020weighted}, the following result is easy to obtain.
%----------------
\begin{lemma}  \label{lem:max-function-bound-morrey}
%------------
Let  $1<q<\infty$ and $0<\lambda<n$.   Then   $\mathcal{M}^{p}$ is bounded from $L^{q,\lambda}(\mathbb{Q}_{p}^{n}) $ to itself.
%--------
\end{lemma}
%------------

Similar to \cite{zhang2014commutators}(or see \cite[Lemmas 3.1 and 3.2]{agcayazi2015note}), we give the following pointwise estimates for $[b,\mathcal{M}_{\alpha}^{p}]$. % when $b\in  \bmo(\mathbb{Q}_{p}^{n})$.
%-----------------
\begin{lemma}  \label{lem:nonlinear-frac-max-pointwise-estimate}
%---------------
   Let $0 \le \alpha<n$    and $f$ be a locally integrable function on $\mathbb{Q}_{p}^{n}$.
%    If $b\in  \bmo(\mathbb{Q}_{p}^{n})$  and $b^{-}\in L^{\infty}(\mathbb{Q}_{p}^{n})$,
%--------------
\begin{enumerate}[label=(\roman*)]
%--------------
\item  If $b $ is any non-negative locally integrable function  on $\mathbb{Q}_{p}^{n}$, then, for any $x\in \mathbb{Q}_{p}^{n}$ such that    $\mathcal{M}_{\alpha}^{p}(f)(x)<\infty $, we have
%----------
\begin{align*}
%-------------
  \big| [b,\mathcal{M}_{\alpha}^{p}](f)(x)  \big| \le \mathcal{M}_{\alpha,b}^{p} (f)(x).
%------------
\end{align*}
%----------
 \vspace{-1em}\label{enumerate:lem-nonlinear-frac-max-pointwise-estimate-1}
%-----------
 \item   If $b $ is any  locally integrable function  on $\mathbb{Q}_{p}^{n}$, then, for any $x\in \mathbb{Q}_{p}^{n}$ such that    $\mathcal{M}_{\alpha}^{p}(f)(x)<\infty $, we have
%----------
\begin{align*}
%------------------
  \big| [b,\mathcal{M}_{\alpha}^{p}](f)(x)  \big| \le \mathcal{M}_{\alpha,b}^{p} (f)(x)+2b^{-}(x)\mathcal{M}_{\alpha}^{p} (f)(x).
%------------
\end{align*}
%----------
    \label{enumerate:thm-nonlinear-frac-max-pointwise-estimate-2}
%-----------
\end{enumerate}
%----------
\end{lemma}
%----------

%-----------------
\begin{proof}
%-----------
\begin{enumerate}[label=(\roman*)]
%----------
\item  For any fixed $x\in \mathbb{Q}_{p}^{n}$ such that    $\mathcal{M}_{\alpha}^{p}(f)(x)<\infty $, noting that $b\ge 0$, then
 %-------
\begin{align*}
%-------
 \big| [b,\mathcal{M}_{\alpha}^{p}](f)(x)  \big| &=  \big| b(x)\mathcal{M}_{\alpha}^{p}(f)(x) -\mathcal{M}_{\alpha}^{p} (bf)(x)  \big|  \\
%  &= \Big| \sup_{\gamma\in \mathbb{Z}  \atop x\in \mathbb{Q}_{p}^{n}} \dfrac{1}{|B_{\gamma} (x)|_{h}^{1-\alpha/n}}  \dint_{B_{\gamma} (x)}  b(x)  |f(y)| \mathd y - \sup_{\gamma\in \mathbb{Z}  \atop x\in \mathbb{Q}_{p}^{n}} \dfrac{1}{|B_{\gamma} (x)|_{h}^{1-\alpha/n}}  \dint_{B_{\gamma} (x)}  b(y)  |f(y)| \mathd y  \Big| \\
  &\le  \sup_{\gamma\in \mathbb{Z}  \atop x\in \mathbb{Q}_{p}^{n}} \dfrac{1}{|B_{\gamma} (x)|_{h}^{1-\alpha/n}}  \dint_{B_{\gamma} (x)} \big| b(x)-b(y) \big| |f(y)| \mathd y  \\
     &=  \mathcal{M}_{\alpha,b}^{p} (f)(x).
%----------
\end{align*}
%------
%-----------
 \item   For any fixed $x\in \mathbb{Q}_{p}^{n}$ such that    $\mathcal{M}_{\alpha}^{p}(f)(x)<\infty $, noting that $|b(x)|-b(x) =2b^{-}(x)$ and $\mathcal{M}_{\alpha}^{p} (bf)(x)=\mathcal{M}_{\alpha}^{p} (|b|f)(x)$,  then we get
 %-------
\begin{align*}
%-------
  \big| [b,\mathcal{M}_{\alpha}^{p}](f)(x)-[|b|,\mathcal{M}_{\alpha}^{p}](f)(x)  \big|
 % &=  \big| b(x)\mathcal{M}_{\alpha}^{p}(f)(x) -\mathcal{M}_{\alpha}^{p} (bf)(x) -|b(x)|\mathcal{M}_{\alpha}^{p}(f)(x) +\mathcal{M}_{\alpha}^{p} (|b|f)(x)   \big|  \\
  &\le  \big| \mathcal{M}_{\alpha}^{p} (bf)(x)-\mathcal{M}_{\alpha}^{p} (|b|f)(x) \big| + \big|2b^{-}(x)\mathcal{M}_{\alpha}^{p}(f)(x)  \big|\\
     &=  2b^{-}(x)\mathcal{M}_{\alpha}^{p}(f)(x).
%----------
\end{align*}
%------
Thus, using \cref{lem:frac-max-almost-every} \labelcref{enumerate:lem-nonlinear-frac-max-pointwise-estimate-1}, we have
%-------
\begin{align*}
%-------
  \big| [b,\mathcal{M}_{\alpha}^{p}](f)(x) \big|
   %&=  \big| [b,\mathcal{M}_{\alpha}^{p}](f)(x)-[|b|,\mathcal{M}_{\alpha}^{p}](f)(x) +[|b|,\mathcal{M}_{\alpha}^{p}](f)(x)  \big| \\
%   &\le   \big| [b,\mathcal{M}_{\alpha}^{p}](f)(x)-[|b|,\mathcal{M}_{\alpha}^{p}](f)(x)  \big| + \big|[|b|,\mathcal{M}_{\alpha}^{p}](f)(x)  \big| \\
  &\le    2b^{-}(x)\mathcal{M}_{\alpha}^{p}(f)(x)+\mathcal{M}_{\alpha,b}^{p} (f)(x).             \tag*{\qedhere}
%----------
\end{align*}
%------
%-----------
\end{enumerate}
%----------
\end{proof}
%------------

The following $p$-adic version of Kolmogorov's inequality  can be obtained by using the relation between the distribution function and the Lebesgue integral  (or refer to \cite[Lemma 2.5]{he2023necessary}),
%-----------------
\begin{lemma}[Kolmogorov's inequality]  \label{lem:Kolmogorov-inequality-p-adic-lem2.5-he2023necessary}
%---------------
 Let $0<r<q<\infty $.  Then, for any $p$-adic ball $B_{\gamma}(x)$ and any measurable funcction $f$, there is a positive constant $C$ such that
 %----------
\begin{align*}
%------------
 |B_{\gamma}(x)|_h^{-1/r}    \|f\|_{L^{r}(B_{\gamma}(x))}  &\le C |B_{\gamma}(x)|_h^{-1/q} \|f\|_{L^{q,\infty}(B_{\gamma}(x))}.
%-------------
\end{align*}
%----------
%----------
\end{lemma}
%----------

%----------
\begin{proof}
%-----------
 % 思想可以参考\cite[Lemma 5.16]{duoandi2001fourier},\cite[Exercise 2.1.5]{grafakos2014classical}.  某些记号可参见\cite{[P.197]grafakos2014modern}\\
For any given $p$-adic ball $B_{\gamma} (x)\subset \mathbb{Q}_{p}^{n}$ and any measurable funcction $f$. Let $t $ be some positive real number which will be determined later, noting $r<q$, by using  the relation between the distribution function and the Lebesgue integral,
% the integration by parts, the $p$-adic version of Chebyshev's inequality,
we have
 %-------
\begin{align*}
%-------
 \dint_{B_{\gamma} (x)}|f(y)|^{r} \mathd y  &=  r \dint_{0}^{\infty} \xi^{r-1} \big|\{y\in B_{\gamma} (x): |f(y)|>\xi\} \big|_{h} \mathd \xi \\
  &\le r \dint_{0}^{\infty} \xi^{r-1} \min \Big \{|B_{\gamma} (x)|_{h}, \frac{1}{\xi^{q}}\|f\|_{L^{q,\infty}(B_{\gamma}(x))}^{q} \Big\} \mathd \xi \\
  &\le   r \dint_{0}^{t} \xi^{r-1}  |B_{\gamma} (x)|_{h}  \mathd \xi + r \dint_{t}^{\infty} \xi^{r-1}   \frac{1}{\xi^{q}}\|f\|_{L^{q,\infty}(B_{\gamma}(x))}^{q}  \mathd \xi   \\
  &=  r\Big( t^{r}|B_{\gamma} (x)|_{h}    + \frac{1}{q-r}  t^{r-q}  \|f\|_{L^{q,\infty}(B_{\gamma}(x))}^{q}\Big).
%----------
\end{align*}
%------
We choose $t= \Big((q-r)|B_{\gamma} (x)|_{h}\Big)^{-1/q} \|f\|_{L^{q,\infty}(B_{\gamma}(x))}$ such that the two terms in the parenthesis above    are equal. Thus, we deduce that
%----------
\begin{align*}
%------------
 |B_{\gamma}(x)|_h^{-1/r}    \|f\|_{L^{r}(B_{\gamma}(x))}  &\le C |B_{\gamma}(x)|_h^{-1/q} \|f\|_{L^{q,\infty}(B_{\gamma}(x))}.
%----------
\end{align*}
%------

 The proof is completed.
%----------
\end{proof}
%------------

Similar to \cite[Lemma 5.3]{guliyev2021commutators}(or see \cite[Theorem 1.10]{agcayazi2015note}), the following gives the pointwise estimate of $\mathcal{M}_{\alpha,b}^{p} $ when $b\in  \bmo(\mathbb{Q}_{p}^{n})$.
%-----------------
\begin{lemma}  \label{lem:frac-max-pointwise-estimate}
%---------------
   Let $0 \le \alpha<n$ and $b\in  \bmo(\mathbb{Q}_{p}^{n})$. Then there exists a positive
constant $C$ such that
%----------
\begin{align*}
%------------
 \mathcal{M}_{\alpha,b}^{p} (f)(x) \le C \|b\|_{p,*} \Big(\mathcal{M}^{p}\big(\mathcal{M}_{\alpha}^{p}(f) \big)(x) + \mathcal{M}_{\alpha}^{p}\big(\mathcal{M}^{p}(f) \big)(x)\Big)
%-------------
\end{align*}
%----------
holds for almost every $x\in \mathbb{Q}_{p}^{n}$ and for all functions from $f\in L_{\loc}^{1}(\mathbb{Q}_{p}^{n})$.
%----------
\end{lemma}
%----------

%-----------------
\begin{proof}
  Let $x\in \mathbb{Q}_{p}^{n}$ and   fix  $p$-adic ball $B_{\gamma} (x)$. We divide  $f= f_{1} +f_{2}$, where
%---------------
\begin{align*}
%-------
  f_{1} (y) &= f(y)  \dchi_{B_{\gamma+1} (x)}(y) =
%------------
 \begin{cases}
 %----------
  f(y)  & \text{if}\ y \in B_{\gamma+1} (x)  \\
0  & \text{if}\  y\notin B_{\gamma+1} (x)
%------------
\end{cases}
%------------
\\ \intertext{and}  %\hspace{1em} \text{and} \hspace{1em}
%------------
  f_{2} (y) &=  f(y)- f_{1} (y)=
%------------
 \begin{cases}
 %----------
  0  & \text{if}\ y \in B_{\gamma+1} (x)  \\
f (y)  & \text{if}\  y\not\in B_{\gamma+1} (x).
%------------
\end{cases}
%------------
\end{align*}
%----------
 Then for any  $y\in \mathbb{Q}_{p}^{n} $, we have
\begin{align*}
\mathcal{M}_{\alpha,b}^{p} (f)(y)
%&=  \mathcal{M}_{\alpha}^{p} \big( (b-b(y))f\big) (y) =\mathcal{M}_{\alpha}^{p} \big( (b-b_{B_{\gamma+1}(x)} +b_{B_{\gamma+1}(x)}-b(y))f\big) (y)      \\
   &\le \mathcal{M}_{\alpha}^{p} \big( (b_{B_{\gamma+1}(x)}-b(y))f\big) (y) + \mathcal{M}_{\alpha}^{p} \big( (b-b_{B_{\gamma+1}(x)})f\big) (y)     \\
   &\le |b_{B_{\gamma+1}(x)}-b(y)|  \mathcal{M}_{\alpha}^{p} (f) (y) +\mathcal{M}_{\alpha}^{p} \big( (b-b_{B_{\gamma+1}(x)})f_{1}\big) (y) +\mathcal{M}_{\alpha}^{p} \big( (b-b_{B_{\gamma+1}(x)})f_{2}\big) (y).
\end{align*}

For $0<\delta<1$, applying   $p$-adic  Lebesgue differentiation theorem (see \cref{lem:cor2.11-kim2009carleson}) and  H\"{o}lder's inequality with exponent $\delta$, we have
%----------
\begin{align} \label{inequ:relation-positive-p-adic-frac-max}
%------------
 \mathcal{M}_{\alpha,b}^{p} (f)(x) \le   \mathcal{M}_{\delta}^{p}\big(\mathcal{M}_{\alpha,b}^{p} (f)\big)(x).
%-------------
\end{align}
%----------

  The following only considers the estimation of  $\mathcal{M}_{\delta}^{p}\big(\mathcal{M}_{\alpha,b}^{p} (f)\big)(x)$. Then
%-----------
\begin{align*}
%-------
 \Big( \frac{1}{|B_{\gamma} (x)|_{h}} \dint_{B_{\gamma} (x)} \big( \mathcal{M}_{\alpha,b}^{p}  (f) (y)\big)^{\delta} \mathrm{d}y \Big)^{\frac{1}{\delta}}
 &\le \Big( \frac{1}{|B_{\gamma} (x)|_{h}} \dint_{B_{\gamma} (x)} \big(|b_{B_{\gamma+1}(x)}-b(y)| \mathcal{M}_{\alpha}^{p} (f) (y)\big)^{\delta} \mathrm{d}y \Big)^{\frac{1}{\delta}}     \\
 &\ \ \ +  \Big( \frac{1}{|B_{\gamma} (x)|_{h}} \dint_{B_{\gamma} (x)} \big(\mathcal{M}_{\alpha}^{p} \big( (b-b_{B_{\gamma+1}(x)})f_{1}\big) (y)\big)^{\delta} \mathrm{d}y \Big)^{\frac{1}{\delta}}     \\
 &\ \ \ + \Big( \frac{1}{|B_{\gamma} (x)|_{h}} \dint_{B_{\gamma} (x)} \big( \mathcal{M}_{\alpha}^{p} \big( (b-b_{B_{\gamma+1}(x)})f_{2}\big) (y)\big)^{\delta} \mathrm{d}y \Big)^{\frac{1}{\delta}}     \\
   &=  I_{1}     +  I_{2}  +I_{3}.
%------------
\end{align*}
%---------

First consider the estimation of $I_{1}$, by appling \cref{lem:cor-5.17-kim2009carleson} and H\"{o}lder's inequality with exponent $1/\delta$, we achieve
%-----------
\begin{align*}
%-------
I_{1}
%&=  \Big( \frac{1}{|B_{\gamma} (x)|_{h}} \dint_{B_{\gamma} (x)} \big(|b_{B_{\gamma+1}(x)}-b(y)| \mathcal{M}_{\alpha}^{p} (f) (y)\big)^{\delta} \mathrm{d}y \Big)^{\frac{1}{\delta}}     \\
 &\le   \Big( \frac{1}{|B_{\gamma} (x)|_{h}} \dint_{B_{\gamma} (x)}  |b_{B_{\gamma+1}(x)}-b(y)|^{\frac{\delta}{1-\delta}} \mathrm{d}y \Big)^{\frac{1-\delta}{\delta}}   \frac{1}{|B_{\gamma} (x)|_{h}} \dint_{B_{\gamma} (x)}   \mathcal{M}_{\alpha}^{p} (f) (y)  \mathrm{d}y    \\
&\le C\|b\|_{p,*}  \mathcal{M}^{p}\big(\mathcal{M}_{\alpha}^{p}(f) \big)(x).
%------------
\end{align*}
%---------

To estimate $I_{2}$, using \cref{lem:frac-max-p-adic-estimate} \labelcref{enumerate:lem-8-he2022characterization-2}  and Kolmogorov's inequality  (\cref{lem:Kolmogorov-inequality-p-adic-lem2.5-he2023necessary}) with $0<\delta<1<n/(n -\alpha) $, we yield
%-----------
\begin{align*}
%-------
 I_{2}
 %&=   \Big( \frac{1}{|B_{\gamma} (x)|_{h}} \dint_{B_{\gamma} (x)} \big(\mathcal{M}_{\alpha}^{p} \big( (b-b_{B_{\gamma+1}(x)})f_{1}\big) (y)\big)^{\delta} \mathrm{d}y \Big)^{\frac{1}{\delta}}     \\
 &=    |B_{\gamma} (x)|_{h}^{-1/\delta}    \big\|\mathcal{M}_{\alpha}^{p} \big( (b-b_{B_{\gamma+1}(x)})f_{1}\big)\big\|_{L^{\delta}(B_{\gamma}(x))}      \\
 &\le C |B_{\gamma}(x)|_h^{-1+\alpha/n} \big\|\mathcal{M}_{\alpha}^{p} \big( (b-b_{B_{\gamma}(x)})f_{1}\big)\big\|_{L^{\frac{n}{n-\alpha},\infty}(B_{\gamma+1}(x))}    \\
 &\le C |B_{\gamma}(x)|_h^{-1+\alpha/n} \big\|  (b-b_{B_{\gamma+1}(x)})f_{1}\big\|_{L^{1}(B_{\gamma+1}(x))}  \\
 &=    \frac{C}{|B_{\gamma} (x)|_{h}^{1-\alpha/n}} \dint_{B_{\gamma+1} (x)}   \big|  b(y)-b_{B_{\gamma+1}(x)} \big| |f(y)| \mathrm{d}y.
%------------
\end{align*}
%---------
Using the generalized  H\"{o}lder's inequality \eqref{equ:generalized-holder-inequality-orlicz-bmo},  \cref{def.p-adic-max-frac-orlicz} and \cref{rem.def-p-adic-max-frac-orlicz} \labelcref{enumeratedef-p-adic-max-frac-orlicz-2}, we get
%-----------
\begin{align*}
%-------
 I_{2}
% &\le   \frac{C}{|B_{\gamma} (x)|_{h}^{1-\alpha/n}} \dint_{B_{\gamma+1} (x)}   \big|  b(y)-b_{B_{\gamma+1}(x)} \big| |f(y)| \mathrm{d}y \\
  &\le   C|B_{\gamma} (x)|_{h}^{\alpha/n}  \|b\|_{\bmo(\mathbb{Q}_{p}^n)} \|f\|_{L(\log L),B_{\gamma}(x)} \\
%  &\le  C \|b\|_{\bmo(\mathbb{Q}_{p}^n)}  \mathcal{M}_{\alpha, L(\log L)}^{p}(f)(x)  \\
   &\le  C \|b\|_{\bmo(\mathbb{Q}_{p}^n)}  \mathcal{M}_{\alpha}^{p}(\mathcal{M}^{p}).
%------------
\end{align*}
%---------

%{\color{red}
%For  $I_{3}$,
%%-----------
%\begin{align*}
%%-------
% I_{3}&=  \Big( \frac{1}{|B_{\gamma} (x)|_{h}} \dint_{B_{\gamma} (x)} \big( \mathcal{M}_{\alpha}^{p} \big( (b-b_{B_{\gamma+1}(x)})f_{2}\big) (y)\big)^{\delta} \mathrm{d}y \Big)^{\frac{1}{\delta}}
%%------------
%\end{align*}
%%---------
%where
%%-----------
%\begin{align*}
%%-------
% \mathcal{M}_{\alpha}^{p} \big( (b-b_{B_{\gamma+1}(x)})f_{2}\big) (y) &=  \sup_{\gamma'\in \mathbb{Z}  \atop y\in B_{\gamma} (x)\subset\mathbb{Q}_{p}^{n}} \dfrac{1}{|B_{\gamma'} (y)|_{h}^{1-\alpha/n}}  \dint_{B_{\gamma'} (y)} |b(z)-b_{B_{\gamma+1}(x)}| |f_{2}(z)| \mathd z,
%%------------
%\end{align*}
%%---------
%Clearly, $ B_{\gamma'} (y) \cap \big(\mathbb{Q}_{p}^{n}\setminus B_{\gamma+1} (x) \big) \neq \emptyset$ for any $y\in B_{\gamma} (x)$. Hence, by \cref{lem:lem-3.1-kim2009q} we have $ B_{\gamma'} (x) \cap \big(\mathbb{Q}_{p}^{n}\setminus B_{\gamma+1} (x) \big) \neq \emptyset$.
%  So we obtain that $\gamma+1 \le \gamma'$. It follows from \cref{lem:lem-3.1-kim2009q} that
% %-----------
%\begin{align*}
%%-------
%   B_{\gamma}(x)=B_{\gamma}(y) \subset B_{\gamma+1}(x)\subset  B_{\gamma'} (x) = B_{\gamma'} (y)
%%------------
%\end{align*}
%%---------
%for any $\gamma'\in \mathbb{Z}$ with $\gamma+1 \le \gamma'$.
%}

For  $I_{3}$.  Reasoning as the proof of (2.10) in \cite{he2023necessary}, due to $I_{3}$  is comparable to $\inf\limits_{y\in B_{\gamma} (x)} \mathcal{M}_{\alpha}^{p} \big( (b-b_{B_{\gamma+1}(x)})f_{2}\big) (y)$ (see \cite[P.1298]{kim2009carleson}), by using the generalized  H\"{o}lder's inequality \eqref{equ:generalized-holder-inequality-orlicz-bmo},  \cref{def.p-adic-max-frac-orlicz} and \cref{rem.def-p-adic-max-frac-orlicz} \labelcref{enumeratedef-p-adic-max-frac-orlicz-2},   we have
%-----------
\begin{align*}
%-------
 I_{3}&\le    \mathcal{M}_{\alpha}^{p} \big( (b-b_{B_{\gamma+1}(x)})f\big) (x)    \\
% &=  \sup_{\gamma\in \mathbb{Z}  \atop x\in \mathbb{Q}_{p}^{n}} \dfrac{1}{|B_{\gamma} (x)|_{h}^{1-\alpha/n}}  \dint_{B_{\gamma} (x)} |b(y)-b_{B_{\gamma+1}(x)}| |f(y)| \mathd y  \\
 &\le C\sup_{\gamma\in \mathbb{Z}  \atop x\in \mathbb{Q}_{p}^{n}} |B_{\gamma} (x)|_{h}^{ \alpha/n}  \|b\|_{\bmo(\mathbb{Q}_{p}^n)} \|f\|_{L(\log L),B_{\gamma}(x)} \\
%  &\le  C \|b\|_{\bmo(\mathbb{Q}_{p}^n)}  \mathcal{M}_{\alpha, L(\log L)}^{p}(f)(x) \\
   &\le  C \|b\|_{\bmo(\mathbb{Q}_{p}^n)}  \mathcal{M}_{\alpha}^{p}(\mathcal{M}^{p}).
%------------
\end{align*}
%---------

Combining the above estimates $I_{1}$, $I_{2}$ and $I_{3}$ together with \eqref{inequ:relation-positive-p-adic-frac-max} yields the desired result.
\end{proof}
%-----------

Finally, we also need the following results (see \cite[Lemma 2.11]{wu2023characterizationC}).
% Similar to the proof of Lemma 2.3 in \cite{zhang2009commutators}, and referring  to  the course of the proof of Theorem 1.4 in \cite{he2023necessary},    using \cref{lem:lem-3.1-kim2009q}  and the properties of $p$-adic ball,  through elementary calculations and derivations, the following assertions can be obtained. % (the ideas and methods were shown in the course of the proof of Theorem 1.4 of \cite{he2023necessary}).
%Hence,  we omit the proofs.    %,   see also \cite{zhang2009commutators}
%--------------------------------
\begin{lemma} \label{lem:frac-max-pointwise-property} %\color{red}
%--------------------------
Let  $b$ be a locally integrable function and $\mathbb{Q}_{p}^{n}$ be an $n$-dimensional $p$-adic vector space. For any fixed  $p$-adic ball $B_{\gamma} (x)\subset \mathbb{Q}_{p}^{n}$.
%-------
 \begin{enumerate}[label=(\roman*)] %\arabicfullwidth,
%------------
\item  If  $0 \le \alpha<n$,     then for all $y\in B_{\gamma}(x)$, we have
%----------
\begin{align*}
%------------------
   \mathcal{M}_{\alpha}^{p}(b\dchi_{B_{\gamma}(x)})(y)  &=  \mathcal{M}_{\alpha,B_{\gamma}(x)}^{p}(b)(y)
%------------
\\ \intertext{and}
%------------
    \mathcal{M}_{\alpha}^{p}(\dchi_{B_{\gamma}(x)})(y)  &=  \mathcal{M}_{\alpha,B_{\gamma}(x)}^{p}(\dchi_{B_{\gamma}(x)})(y)=|B_{\gamma}(x)|_{h}^{\alpha/n}.
%------------
\end{align*}
%----------
    \vspace{-1em}\label{enumerate:lem-frac-max-pointwise-property-1}
%-------
\item  Then for any  $y\in B_{\gamma} (x)$, we have
%---------------
\begin{align*}
%-------
    |b_{B_{\gamma}(x)}|   &\le  |B_{\gamma}(x)|_{h}^{-\alpha/n}\mathcal{M}_{\alpha,B_{\gamma}(x)}^{p}(b)(y).
%----------
\end{align*}
%-------
 \vspace{-1em}\label{enumerate:lem-frac-max-pointwise-property-2}
%-------
\item  Let $E=\{y\in B_{\gamma}(x): b(y)\le b_{B_{\gamma}(x)}\}$ and $F=  B_{\gamma}(x)\setminus E =\{y\in B_{\gamma}(x): b(y)> b_{B_{\gamma}(x)}\}$. Then the following equality is trivially true
%----------
\begin{align*}
%-------
    \dint_{E} |b(y)-b_{B_{\gamma}(x)}| \mathd y  &=  \dint_{F} |b(y)-b_{B_{\gamma}(x)}| \mathd y.
%----------
\end{align*}
%-------
\label{enumerate:lem-frac-max-pointwise-property-3}
%-------
%----------
\end{enumerate}
%------------
\end{lemma}
%-------------

%------------------------
\section{Proofs of the main results} %Main results and their proofs} %  the principal results
\label{sec:result-proof}

Now we give the proofs of the \cref{thm:nonlinear-frac-max-var-bmo,thm:frac-max-var-bmo,thm:nonlinear-frac-max-bmo-morrey,thm:frac-max-bmo-morrey}.

\subsection{Proof of \cref{thm:nonlinear-frac-max-var-bmo}}

To prove \cref{thm:nonlinear-frac-max-var-bmo}, we first prove the following lemma.

%-----------------
\begin{lemma}  \label{lem:frac-max-bmo-norm}
%---------------
   Let $0 < \alpha<n$        and $b$ be a locally integrable function on $\mathbb{Q}_{p}^{n}$. If there exists a positive constant $C$ such that
%-----------
\begin{align} \label{inequ:lem-frac-max-bmo-norm}
%-----------
\sup_{\gamma\in \mathbb{Z} \atop x\in \mathbb{Q}_{p}^{n}}   \dfrac{\Big\| \big(b -|B_{\gamma} (x)|_{h}^{-\alpha/n}\mathcal{M}_{\alpha,B_{\gamma} (x)}^{p} (b) \big) \dchi_{B_{\gamma} (x)} \Big\|_{L^{q(\cdot)}(\mathbb{Q}_{p}^{n}) }}{\|\dchi_{B_{\gamma} (x)}\|_{L^{q(\cdot)}(\mathbb{Q}_{p}^{n}) }} \le C
%-----------------
\end{align}
%-----------
 for some  $ q(\cdot)\in   \mathscr{B}(\mathbb{Q}_{p}^{n}) $, then $b\in  \bmo(\mathbb{Q}_{p}^{n})$.
%----------
\end{lemma}
%----------

%-----------------
\begin{proof}
 Some ideas are taken from \cite{bastero2000commutators,zhang2009commutators} and  \cite{zhang2014commutators}. Reasoning as the proof of (4.4) in \cite{zhang2014commutators},  for any fixed  $p$-adic ball $B_{\gamma} (x)\subset \mathbb{Q}_{p}^{n}$, we have (see   \cref{lem:frac-max-pointwise-property} \labelcref{enumerate:lem-frac-max-pointwise-property-2})
%---------------
\begin{align*}
%-------
    |b_{B_{\gamma}(x)}|   &\le  |B_{\gamma}(x)|_{h}^{-\alpha/n}\mathcal{M}_{\alpha,B_{\gamma}(x)}^{p}(b)(y), \qquad  \forall~ y\in B_{\gamma} (x).
%----------
\end{align*}
%-------
Let $E=\{y\in B_{\gamma}(x): b(y)\le b_{B_{\gamma}(x)}\}$  and $F=  B_{\gamma}(x)\setminus E =\{y\in B_{\gamma}(x): b(y)> b_{B_{\gamma}(x)}\}$, then for any $ y\in E\subset B_{\gamma} (x)$, we have $b(y)\le b_{B_{\gamma}(x)}\le  |b_{B_{\gamma}(x)}| \le  |B_{\gamma}(x)|_{h}^{-\alpha/n}\mathcal{M}_{\alpha,B_{\gamma}(x)}^{p}(b)(y)$.
It is clear that
%---------------
\begin{align*}
%-------
    |b(y)-b_{B_{\gamma}(x)}|   &\le \Big| b(y) - |B_{\gamma}(x)|_{h}^{-\alpha/n}\mathcal{M}_{\alpha,B_{\gamma}(x)}^{p}(b)(y)\Big|, \qquad  \forall~ y\in E.
%----------
\end{align*}
%-------

Therefore, by using    \cref{lem:frac-max-pointwise-property} \labelcref{enumerate:lem-frac-max-pointwise-property-3},   we get
%-----------------
\begin{align*}
%-------
 \dfrac{1}{|B_{\gamma} (x)|_{h} } \dint_{B_{\gamma} (x)} \big| b(y)-b_{B_{\gamma} (x)}) \big| \mathd y &=  \dfrac{1}{|B_{\gamma} (x)|_{h} } \dint_{E\cup F}  \big| b(y)-b_{B_{\gamma} (x)}) \big|  \mathd y  \\
 &=  \dfrac{2}{|B_{\gamma} (x)|_{h} } \dint_{E}  \big| b(y)-b_{B_{\gamma} (x)}) \big|  \mathd y  \\
  &\le  \dfrac{2}{|B_{\gamma} (x)|_{h} } \dint_{E}  \Big| b(y)- |B_{\gamma}(x)|_{h}^{-\alpha/n}\mathcal{M}_{\alpha,B_{\gamma}(x)}^{p}(b)(y) \Big|  \mathd y  \\
  &\le  \dfrac{2}{|B_{\gamma} (x)|_{h} } \dint_{B_{\gamma} (x)}  \Big| b(y)- |B_{\gamma}(x)|_{h}^{-\alpha/n}\mathcal{M}_{\alpha,B_{\gamma}(x)}^{p}(b)(y) \Big|  \mathd y.
%----------
\end{align*}
%------------
 Using   generalized H\"{o}lder's inequality  (see \cref{lem:holder-inequality-p-adic} \labelcref{enumerate:holder-p-adic-variable-1}),   \labelcref{inequ:lem-frac-max-bmo-norm} and   \cref{lem:norm-characteristic-p-adic} \labelcref{enumerate:charact-norm-conjugate-p-adic-variable}, we obtain
%-----------------
\begin{align*}
%-------
 \dfrac{1}{|B_{\gamma} (x)|_{h}} \dint_{B_{\gamma} (x)} \big| b(y)-b_{B_{\gamma} (x)}) \big| \mathd y
%  &\le  \dfrac{2}{|B_{\gamma} (x)|_{h}} \dint_{B_{\gamma} (x)}  \Big| b(y)- |B_{\gamma}(x)|_{h}^{-\alpha/n}\mathcal{M}_{\alpha,B_{\gamma}(x)}^{p}(b)(y) \Big|  \mathd y  \\
 &\le  \dfrac{C}{|B_{\gamma} (x)|_{h}} \Big\| \big(b -|B_{\gamma} (x)|_{h}^{-\alpha/n}\mathcal{M}_{\alpha,B_{\gamma} (x)}^{p} (b) \big) \dchi_{B_{\gamma} (x)} \Big\|_{L^{q(\cdot)}(\mathbb{Q}_{p}^{n}) }   \\
 &\qquad \times \|\dchi_{B_{\gamma} (x)}\|_{L^{q'(\cdot)}(\mathbb{Q}_{p}^{n}) } \\
   &\le \dfrac{C}{|B_{\gamma} (x)|_{h}}  \|  \dchi_{B_{\gamma} (x)}  \|_{L^{q(\cdot)}(\mathbb{Q}_{p}^{n}) }  \|\dchi_{B_{\gamma} (x)}\|_{L^{q'(\cdot)}(\mathbb{Q}_{p}^{n}) }  \\
&\le C.
%----------
\end{align*}
%------------
So, the proof is completed by applying \cref{def.5.3-bmo-space-kim2009carleson}.   %{lem:cor-5.17-kim2009carleson}

\end{proof}
%------------

%-----------------
\begin{figure}[!ht] \centering
%-------
 \begin{tikzpicture}[vertex/.style = {shape=circle,draw,minimum size=1em},]
  % 定义单位圆半径(长度可任意)
  \edef\r{1.2cm}
 % 定义基础点(尽可能少)
  \tkzDefPoints{0/-2/M, 0/2/N}
  % 计算其它点
  % 定义垂直平分线并求垂足
  \tkzDefLine[mediator](M,N) \tkzGetPoints{x}{x'}
  \tkzInterLL(M,N)(x,x') \tkzGetPoint{O}
  % 确定正五边形连长
  \tkzInterLC[R](x',x)(O,\r) \tkzGetPoints{b}{a}
  % 定义正五边形，求得五个顶点
  \tkzDefRegPolygon[side,sides=5,name=P](a,b)

% 绘制点
  \tkzSetUpPoint[size = 4,fill = black!50, color = blue]
  \tkzDrawPoints(P1,P2,P3,P4,P5)

% 标注各点名称
  \tkzLabelPoint[above left,vertex](P4){$4$}
  \tkzLabelPoint[above right,vertex](P3){$3$}
  \tkzLabelPoint[below right,vertex](P2) {$2$}
  \tkzLabelPoint[below left,vertex](P1) {$1$}
    \tkzLabelPoint[above left,vertex](P5){$5$}

%  % 绘制
%  % 绘制基础正五边形和需要线段
%  \tkzDrawPolygon[dashed,fill=yellow!80,draw=blue,fill opacity=0.5](P1,P2,P3,P4,P5)
%  \tkzDrawPolygon[dashed,fill=white!80,draw=blue,fill opacity=0.5](P1,P2,P3,P4)
\tkzDrawSegment[-latex,very thick, blue!80!](P1,P2)  % ultra
\tkzDrawSegment[-latex,very thick, blue!80!](P3,P4)
\tkzDrawSegment[-latex,very thick, blue!80!](P4,P1)
%\tkzDrawSegment[-latex,very thick, blue!80!](P2,P5)
%   \tkzDrawSegments[-latex,dotted,very thick, black!80](P2,P3   P5,P4)
     \tkzDrawSegments[-latex,dotted,very thick, black!80](P2,P3 P2,P5 P5,P4)

%  % 标注线段长度
  \tkzLabelSegment[above,sloped,midway,font=\tiny,blue](P3,P4){$w_{34}$}
  \tkzLabelSegment[below,sloped,midway,font=\tiny,blue](P1,P2){$w_{12}$}
    \tkzLabelSegment[above right,sloped,midway,rotate=-70,shift={(4pt,10pt)},font=\tiny,blue](P1,P4){$w_{41}$}
  \tkzLabelSegment[below right,sloped,midway,rotate=-75,font=\tiny](P2,P3){$w_{23}$}
  \tkzLabelSegment[above,sloped,midway,shift={(20pt,0pt)},font=\tiny](P2,P5){$w_{25}$}
%  \tkzLabelSegment[above,sloped,midway,shift={(20pt,0pt)},font=\tiny,blue](P2,P5){$w_{25}$}
   \tkzLabelSegment[above,sloped,midway, font=\tiny](P5,P4){$w_{54}$}
\end{tikzpicture}
%------------
\vskip 6pt
\caption{Proof structure of \cref{thm:nonlinear-frac-max-var-bmo}  \\ where $w_{ij}$ denotes $i\Longrightarrow j$}\label{fig:proof-structure-frac-max-lip-1}
\end{figure}
%----------

%--------------------
\begin{proof}[Proof of  \cref{thm:nonlinear-frac-max-var-bmo}]
%---------------
Since the implications \labelcref{enumerate:thm-nonlinear-frac-max-var-bmo-2} $\xLongrightarrow{\ \  }$ \labelcref{enumerate:thm-nonlinear-frac-max-var-bmo-3} and \labelcref{enumerate:thm-nonlinear-frac-max-var-bmo-5} $\xLongrightarrow{\ \ }$ \labelcref{enumerate:thm-nonlinear-frac-max-var-bmo-4} follow readily,
and \labelcref{enumerate:thm-nonlinear-frac-max-var-bmo-2} $\xLongrightarrow{\ \  }$ \labelcref{enumerate:thm-nonlinear-frac-max-var-bmo-5} is similar to \labelcref{enumerate:thm-nonlinear-frac-max-var-bmo-3} $\xLongrightarrow{\ \  }$ \labelcref{enumerate:thm-nonlinear-frac-max-var-bmo-4},
we only need to prove \labelcref{enumerate:thm-nonlinear-frac-max-var-bmo-1} $\xLongrightarrow{\ \  }$ \labelcref{enumerate:thm-nonlinear-frac-max-var-bmo-2}, \labelcref{enumerate:thm-nonlinear-frac-max-var-bmo-3} $\xLongrightarrow{\ \  }$ \labelcref{enumerate:thm-nonlinear-frac-max-var-bmo-4} and  \labelcref{enumerate:thm-nonlinear-frac-max-var-bmo-4} $\xLongrightarrow{\ \  }$  \labelcref{enumerate:thm-nonlinear-frac-max-var-bmo-1}
%and \labelcref{enumerate:thm-nonlinear-frac-max-var-bmo-2} $\xLongrightarrow{\ \  }$ \labelcref{enumerate:thm-nonlinear-frac-max-var-bmo-5}
(see \Cref{fig:proof-structure-frac-max-lip-1} for the proof structure).

 \labelcref{enumerate:thm-nonlinear-frac-max-var-bmo-1} $\xLongrightarrow{\ \  }$ \labelcref{enumerate:thm-nonlinear-frac-max-var-bmo-2}:
 Let $b\in  \bmo(\mathbb{Q}_{p}^{n})$ and $b^{-}\in L^{\infty}(\mathbb{Q}_{p}^{n})$. We need to prove that $ [b,\mathcal{M}_{\alpha}^{p}] $ is bounded from $L^{r(\cdot)}(\mathbb{Q}_{p}^{n})$ to $L^{q(\cdot)}(\mathbb{Q}_{p}^{n})$ for all $r(\cdot), q(\cdot)\in   \mathscr{C}^{\log}(\mathbb{Q}_{p}^{n}) $ with $r(\cdot)\in \mathscr{P}(\mathbb{Q}_{p}^{n})$, $ r_{+}<\frac{n}{\alpha}$ and $1/q(\cdot) = 1/r(\cdot) -\alpha/n$.
  For such $r(\cdot)$ and any $f\in L^{r(\cdot)}(\mathbb{Q}_{p}^{n})$, it
follows from \cref{lem:frac-max-almost-every} that $\mathcal{M}_{\alpha}^{p}(f)(x)<\infty $ for almost everywhere $x\in \mathbb{Q}_{p}^{n}$. By  \cref{lem:nonlinear-frac-max-pointwise-estimate}  \labelcref{enumerate:thm-nonlinear-frac-max-pointwise-estimate-2} and \cref{lem:frac-max-pointwise-estimate}, %\cref{def.p-adic-max-frac-orlicz} and \cref{rem.def-p-adic-max-frac-orlicz},
we obtain
%----------
\begin{align*}
%-------------
  \big| [b,\mathcal{M}_{\alpha}^{p}](f)(x)  \big| &\le \mathcal{M}_{\alpha,b}^{p} (f)(x)+2b^{-}(x)\mathcal{M}_{\alpha}^{p} (f)(x)  \\
  &\le C \|b\|_{p,*} \Big(\mathcal{M}^{p}\big(\mathcal{M}_{\alpha}^{p}(f) \big)(x) + \mathcal{M}_{\alpha}^{p}\big(\mathcal{M}^{p}(f) \big)(x)\Big)+ +2b^{-}(x)\mathcal{M}_{\alpha}^{p} (f)(x).
 % C\|b^{-}\|_{L^{\infty}(\mathbb{Q}_{p}^{n})}  \mathcal{M}_{\alpha}^{p}\big(\mathcal{M}^{p}(f) \big)(x)  %\\
%  &\le C \|b\|_{p,*}\mathcal{M}^{p}\big(\mathcal{M}_{\alpha}^{p}(f) \big)(x) +\Big(\|b\|_{p,*} + \|b^{-}\|_{L^{\infty}(\mathbb{Q}_{p}^{n})} \Big) \mathcal{M}_{\alpha}^{p}\big(\mathcal{M}^{p}(f) \big)(x).
%--------------------
\end{align*}
%----------

 Then, statement \labelcref{enumerate:thm-nonlinear-frac-max-var-bmo-2} follows from \cref{lem.thm-5.2-variable-max-bounded} and  \cref{lem:frac-max-p-adic-estimate} \labelcref{enumerate:thm-4-chacon2021fractional}.

\labelcref{enumerate:thm-nonlinear-frac-max-var-bmo-3} $\xLongrightarrow{\ \  }$ \labelcref{enumerate:thm-nonlinear-frac-max-var-bmo-4}:
For any fixed  $p$-adic ball $B_{\gamma} (x)\subset \mathbb{Q}_{p}^{n}$ and any $y\in B_{\gamma} (x)$, it follows from   \cref{lem:frac-max-pointwise-property} \labelcref{enumerate:lem-frac-max-pointwise-property-1} that
%----------
\begin{align*}
%------------------
   \mathcal{M}_{\alpha}^{p}(b\dchi_{B_{\gamma}(x)})(y)   =  \mathcal{M}_{\alpha,B_{\gamma}(x)}^{p}(b)(y)
%------------
 \ \text{and} \
%------------
    \mathcal{M}_{\alpha}^{p}(\dchi_{B_{\gamma}(x)})(y)   =  \mathcal{M}_{\alpha,B_{\gamma}(x)}^{p}(\dchi_{B_{\gamma}(x)})(y)=|B_{\gamma}(x)|_{h}^{\alpha/n}.
%------------
\end{align*}
%----------
Then, for  any $y\in B_{\gamma} (x)$, we have
%---------------
\begin{align*}
%-------
    b(y)-|B_{\gamma} (x)|_{h}^{-\alpha/n} \mathcal{M}_{\alpha,B_{\gamma} (x)}^{p} (b)(y)   &= |B_{\gamma} (x)|_{h}^{-\alpha/n} \big( b(y) |B_{\gamma} (x)|_{h}^{\alpha/n} -\mathcal{M}_{\alpha,B_{\gamma} (x)}^{p} (b)(y) \big)    \\
   &=    |B_{\gamma} (x)|_{h}^{-\alpha/n} \big( b(y)  \mathcal{M}_{\alpha}^{p}(\dchi_{B_{\gamma}(x)})(y) -\mathcal{M}_{\alpha}^{p}(b\dchi_{B_{\gamma}(x)})(y) \big)     \\
    &=    |B_{\gamma} (x)|_{h}^{-\alpha/n}  [b,\mathcal{M}_{\alpha}^{p}] (\dchi_{B_{\gamma}(x)})(y).
%----------
\end{align*}
%------------
 Thus, for any $y\in \mathbb{Q}_{p}^{n}$, we get
%---------------
\begin{align*}
%-------
   \big( b(y)-|B_{\gamma} (x)|_{h}^{-\alpha/n} \mathcal{M}_{\alpha,B_{\gamma} (x)}^{p} (b)(y) \big) \dchi_{B_{\gamma}(x)}(y) &=     |B_{\gamma} (x)|_{h}^{-\alpha/n}  [b,\mathcal{M}_{\alpha}^{p}] (\dchi_{B_{\gamma}(x)})(y) \dchi_{B_{\gamma}(x)}(y).
%----------
\end{align*}
%------------
 By using assertion  \labelcref{enumerate:thm-nonlinear-frac-max-var-bmo-3}  and   \cref{lem:norm-characteristic-p-adic} \labelcref{enumerate:charact-norm-fraction-p-adic-variable}, we have
%----------------
\begin{align*}
%-------
 \Big\| \big(b -|B_{\gamma} (x)|_{h}^{-\alpha/n}\mathcal{M}_{\alpha,B_{\gamma} (x)}^{p} (b) \big) \dchi_{B_{\gamma} (x)} \Big\|_{L^{q(\cdot)}(\mathbb{Q}_{p}^{n}) }
% &\le  |B_{\gamma} (x)|_{h}^{-\alpha/n} \big\| [b,\mathcal{M}_{\alpha}^{p}] (\dchi_{B_{\gamma}(x)}) \big\|_{L^{q(\cdot)}(\mathbb{Q}_{p}^{n}) }  \\
&\le C |B_{\gamma} (x)|_{h}^{-\alpha/n} \big\|  \dchi_{B_{\gamma}(x)} \big\|_{L^{r(\cdot)}(\mathbb{Q}_{p}^{n}) }  \\
% &\le C |B_{\gamma} (x)|_{h}^{-\alpha/n}   |B_{\gamma} (x)|_{h}^{\alpha/n}  \big\|  \dchi_{B_{\gamma}(x)} \big\|_{L^{q(\cdot)}(\mathbb{Q}_{p}^{n}) }   \\
 &\le C     \big\|  \dchi_{B_{\gamma}(x)} \big\|_{L^{q(\cdot)}(\mathbb{Q}_{p}^{n}) },
%----------
\end{align*}
%------------
which gives \labelcref{inequ:thm-nonlinear-frac-max-var-bmo-4} since   $B_{\gamma} (x)$ is arbitrary and $C$ is independent of $B_{\gamma} (x)$.

\labelcref{enumerate:thm-nonlinear-frac-max-var-bmo-4} $\xLongrightarrow{\ \  }$  \labelcref{enumerate:thm-nonlinear-frac-max-var-bmo-1}:
By  \cref{lem:thm1.3-max-bmo-he2023necessary}, it suffices to prove
%--------
\begin{align} \label{inequ:proof-lem-non-negative-max-lip-41}
%-----------
 \sup_{\gamma\in \mathbb{Z} \atop x\in \mathbb{Q}_{p}^{n}} \dfrac{1}{|B_{\gamma} (x)|_{h}} \dint_{B_{\gamma} (x)} \Big|b(y)- \mathcal{M}_{B_{\gamma} (x)}^{p} (b)(y) \Big|  \mathd y <\infty.
%-----------------
\end{align}
%--------

  For any fixed  $p$-adic ball $B_{\gamma} (x)\subset \mathbb{Q}_{p}^{n}$, we have
%---------------
\begin{align} \label{inequ:proof-lem-non-negative-max-lip-41-1}
%-------
\begin{split}
%-------
   \dfrac{1}{|B_{\gamma} (x)|_{h} } &\dint_{B_{\gamma} (x)} \Big|b(y)- \mathcal{M}_{B_{\gamma} (x)}^{p} (b)(y) \Big|  \mathd y    \\
    &\le  \dfrac{1}{|B_{\gamma} (x)|_{h} } \dint_{B_{\gamma} (x)} \Big|b(y)- |B_{\gamma} (x)|_{h}^{-\alpha/n}\mathcal{M}_{\alpha,B_{\gamma} (x)}^{p} (b) (y)  \Big|  \mathd y     \\
    &\qquad +  \dfrac{1}{|B_{\gamma} (x)|_{h} } \dint_{B_{\gamma} (x)} \Big|  |B_{\gamma} (x)|_{h}^{-\alpha/n}\mathcal{M}_{\alpha,B_{\gamma} (x)}^{p} (b) (y)-\mathcal{M}_{B_{\gamma} (x)}^{p} (b)(y) \Big|  \mathd y    \\
     &:=  I_{1}+I_{2}.
%-------
\end{split}
%----------
\end{align}
%-------
For $ I_{1}$, by applying statement \labelcref{enumerate:thm-nonlinear-frac-max-var-bmo-4},  generalized H\"{o}lder's inequality (\cref{lem:holder-inequality-p-adic} \labelcref{enumerate:holder-p-adic-variable-1}) and  \cref{lem:norm-characteristic-p-adic} \labelcref{enumerate:charact-norm-conjugate-p-adic-variable}, we get
%---------------
\begin{align*}
%-------
 I_{1}
 %&=   \dfrac{1}{|B_{\gamma} (x)|_{h} } \dint_{B_{\gamma} (x)} \Big|b(y)- |B_{\gamma} (x)|_{h}^{-\alpha/n}\mathcal{M}_{\alpha,B_{\gamma} (x)}^{p} (b) (y)  \Big|  \mathd y     \\
 &\le  \dfrac{C}{|B_{\gamma} (x)|_{h} } \Big\| \big(b -|B_{\gamma} (x)|_{h}^{-\alpha/n}\mathcal{M}_{\alpha,B_{\gamma} (x)}^{p} (b) \big) \dchi_{B_{\gamma} (x)} \Big\|_{L^{q(\cdot)}(\mathbb{Q}_{p}^{n}) }  \|\dchi_{B_{\gamma} (x)}\|_{L^{q'(\cdot)}(\mathbb{Q}_{p}^{n}) } \\
  &\le   C \dfrac{ \Big\| \big(b -|B_{\gamma} (x)|_{h}^{-\alpha/n}\mathcal{M}_{\alpha,B_{\gamma} (x)}^{p} (b) \big) \dchi_{B_{\gamma} (x)} \Big\|_{L^{q(\cdot)}(\mathbb{Q}_{p}^{n}) }}{\|\dchi_{B_{\gamma} (x)}\|_{L^{q(\cdot)}(\mathbb{Q}_{p}^{n}) }}   \\
&\le C,
%----------
\end{align*}
%-------
where the constant $C$ is independent of $B_{\gamma} (x)$.

 Now we consider $ I_{2}$. For all  $y\in B_{\gamma} (x)$,  it follows from \cref{lem:frac-max-pointwise-property}  that
 %----------
\begin{align*}
%------------------
   \mathcal{M}_{\alpha}^{p}(\dchi_{B_{\gamma}(x)})(y)   =  |B_{\gamma}(x)|_{h}^{\alpha/n} \ \text{and}\  \mathcal{M}_{\alpha}^{p}(b\dchi_{B_{\gamma}(x)})(y)   =  \mathcal{M}_{\alpha,B_{\gamma}(x)}^{p}(b)(y),
%------------
\\ \intertext{and}
%------------
    \mathcal{M}^{p}(\dchi_{B_{\gamma}(x)})(y)   = \dchi_{B_{\gamma}(x)}(y)   =  1 \ \text{and}\  \mathcal{M}^{p}(b\dchi_{B_{\gamma}(x)})(y)   =  \mathcal{M}_{B_{\gamma}(x)}^{p}(b)(y).
%------------
\end{align*}
%----------
Then, for any $y\in B_{\gamma} (x)$,  we get
%---------------
\begin{align} \label{inequ:proof-lem-non-negative-max-lip-41-2}
%-------
\begin{split}
%-------
  &\Big|  |B_{\gamma} (x)|_{h}^{-\alpha/n}\mathcal{M}_{\alpha,B_{\gamma} (x)}^{p} (b) (y)-\mathcal{M}_{B_{\gamma} (x)}^{p} (b)(y) \Big|     \\
    &\le  |B_{\gamma} (x)|_{h}^{-\alpha/n}  \Big|  \mathcal{M}_{\alpha,B_{\gamma} (x)}^{p} (b) (y)- |B_{\gamma} (x)|_{h}^{\alpha/n} |b(y)|  \Big|  + \Big|  |b(y)|-\mathcal{M}_{B_{\gamma} (x)}^{p} (b)(y) \Big|    \\
    &\le  |B_{\gamma} (x)|_{h}^{-\alpha/n}  \Big|   \mathcal{M}_{\alpha}^{p}(b\dchi_{B_{\gamma}(x)})(y)-  |b(y)| \mathcal{M}_{\alpha}^{p}(\dchi_{B_{\gamma}(x)})(y) \Big|  \\
    &\qquad + \Big|  |b(y)|\mathcal{M}^{p}(\dchi_{B_{\gamma}(x)})(y) -\mathcal{M}^{p}(b\dchi_{B_{\gamma}(x)})(y)  \Big|    \\
  &\le  |B_{\gamma} (x)|_{h}^{-\alpha/n}  \Big|[|b|, \mathcal{M}_{\alpha}^{p}](\dchi_{B_{\gamma}(x)})(y) \Big|  + \Big| [ |b|, \mathcal{M}^{p}] (\dchi_{B_{\gamma}(x)})(y)   \Big|.
%----------
\end{split}
%-------
\end{align}
%-------

Since  $ q(\cdot)\in   \mathscr{B}(\mathbb{Q}_{p}^{n}) $ follows at once from  \cref{lem.thm-5.2-variable-max-bounded} and statement \labelcref{enumerate:thm-nonlinear-frac-max-var-bmo-4}. Then statement \labelcref{enumerate:thm-nonlinear-frac-max-var-bmo-4} along with \cref{lem:frac-max-bmo-norm} gives  $b\in  \bmo(\mathbb{Q}_{p}^{n})$, which implies $|b|\in  \bmo(\mathbb{Q}_{p}^{n})$.
Thus, we can apply  \cref{lem:nonlinear-frac-max-pointwise-estimate} to $[|b|, \mathcal{M}_{\alpha}^{p}]$ and $[ |b|, \mathcal{M}^{p}]$ due to  $|b|\in  \bmo(\mathbb{Q}_{p}^{n})$.

By using  \cref{lem:nonlinear-frac-max-pointwise-estimate}, \cref{lem:frac-max-pointwise-estimate} and \cref{lem:frac-max-pointwise-property} \labelcref{enumerate:lem-frac-max-pointwise-property-1}, for any $y\in B_{\gamma} (x)$,  we have
%----------
\begin{align*}
%------------------
  \Big|[|b|, \mathcal{M}_{\alpha}^{p}](\dchi_{B_{\gamma}(x)})(y) \Big| &\le \mathcal{M}_{\alpha,|b|}^{p} (\dchi_{B_{\gamma}(x)})(y)
  \le C \|b\|_{p,*} \Big(\mathcal{M}^{p}\big(\mathcal{M}_{\alpha}^{p}(\dchi_{B_{\gamma}(x)}) \big)(y) + \mathcal{M}_{\alpha}^{p}\big(\mathcal{M}^{p}(\dchi_{B_{\gamma}(x)}) \big)(y)\Big)    \\
  &\le C \|b\|_{p,*}  |B_{\gamma} (x)|_{h}^{\alpha/n}
%------------
\\ \intertext{and}
%------------
 \Big| [ |b|, \mathcal{M}^{p}] (\dchi_{B_{\gamma}(x)})(y)   \Big| &\le \mathcal{M}_{|b|}^{p} (\dchi_{B_{\gamma}(x)})(y)
  \le C \|b\|_{p,*} \mathcal{M}^{p}\big(\mathcal{M}^{p}(\dchi_{B_{\gamma}(x)}) \big)(y)  =C \|b\|_{p,*}.
%------------
\end{align*}
%----------

Hence, it follows from  \labelcref{inequ:proof-lem-non-negative-max-lip-41-2} that
%---------------
\begin{align*}
%-------
 I_{2}
 %&=   \dfrac{1}{|B_{\gamma} (x)|_{h}} \dint_{B_{\gamma} (x)} \Big|  |B_{\gamma} (x)|_{h}^{-\alpha/n}\mathcal{M}_{\alpha,B_{\gamma} (x)}^{p} (b) (y)-\mathcal{M}_{B_{\gamma} (x)}^{p} (b)(y) \Big|  \mathd y    \\
 &\le \dfrac{C}{|B_{\gamma} (x)|_{h}^{1+ \alpha/n}} \dint_{B_{\gamma} (x)}  \Big|[|b|, \mathcal{M}_{\alpha}^{p}](\dchi_{B_{\gamma}(x)})(y) \Big| \mathd y   +\dfrac{C}{|B_{\gamma} (x)|_{h}} \dint_{B_{\gamma} (x)} \Big|  [ |b|, \mathcal{M}^{p}] (\dchi_{B_{\gamma}(x)})(y)   \Big|  \mathd y    \\
     &\le C \|b\|_{p,*}.
%----------
\end{align*}
%-------

Putting the above estimates for $I_{1}$ and $I_{2}$ into \labelcref{inequ:proof-lem-non-negative-max-lip-41-1}, we obtain \labelcref{inequ:proof-lem-non-negative-max-lip-41}.

%\labelcref{enumerate:thm-nonlinear-frac-max-var-bmo-2} $\xLongrightarrow{\ \  }$ \labelcref{enumerate:thm-nonlinear-frac-max-var-bmo-5}:
%Assume statement \labelcref{enumerate:thm-nonlinear-frac-max-var-bmo-2} is true. Reasoning as in the proof of \labelcref{enumerate:thm-nonlinear-frac-max-var-bmo-3} $\xLongrightarrow{\ \  }$ \labelcref{enumerate:thm-nonlinear-frac-max-var-bmo-4}, we have

 This completes the proof of \cref{thm:nonlinear-frac-max-var-bmo}.
%--------------------------------
\end{proof}
%-------------------

\subsection{Proof  of \cref{thm:frac-max-var-bmo} }

%--------------------
\begin{proof}[Proof of \cref{thm:frac-max-var-bmo}]
%---------------
Similar to prove \cref{thm:nonlinear-frac-max-var-bmo},
%---------------
%Since the implications \labelcref{enumerate:thm-frac-max-var-bmo-2} $\xLongrightarrow{\ \  }$ \labelcref{enumerate:thm-frac-max-var-bmo-3} and \labelcref{enumerate:thm-frac-max-var-bmo-5} $\xLongrightarrow{\ \ }$ \labelcref{enumerate:thm-frac-max-var-bmo-4} follow readily,
%and \labelcref{enumerate:thm-frac-max-var-bmo-2} $\xLongrightarrow{\ \  }$ \labelcref{enumerate:thm-frac-max-var-bmo-5} is similar to \labelcref{enumerate:thm-frac-max-var-bmo-3} $\xLongrightarrow{\ \  }$ \labelcref{enumerate:thm-frac-max-var-bmo-4},
 we only need to prove the implications \labelcref{enumerate:thm-frac-max-var-bmo-1} $\xLongrightarrow{\ \  }$ \labelcref{enumerate:thm-frac-max-var-bmo-2}, \labelcref{enumerate:thm-frac-max-var-bmo-3} $\xLongrightarrow{\ \  }$ \labelcref{enumerate:thm-frac-max-var-bmo-4}
and \labelcref{enumerate:thm-frac-max-var-bmo-4} $\xLongrightarrow{\ \  }$  \labelcref{enumerate:thm-frac-max-var-bmo-1} (the proof structure is also shown in \Cref{fig:proof-structure-frac-max-lip-1}).

 \labelcref{enumerate:thm-frac-max-var-bmo-1} $\xLongrightarrow{\ \  }$ \labelcref{enumerate:thm-frac-max-var-bmo-2}:
%------------
%For any $p$-adic ball $B_{\gamma}(x) \subset \mathbb{Q}_{p}^{n}$,
Since $b\in  \bmo(\mathbb{Q}_{p}^{n})$, $r(\cdot), q(\cdot)\in   \mathscr{C}^{\log}(\mathbb{Q}_{p}^{n}) $ with $r(\cdot)\in \mathscr{P}(\mathbb{Q}_{p}^{n})$, $ r_{+}<\frac{n}{\alpha }$ and $1/q(\cdot) = 1/r(\cdot) -\alpha/n$. Then,  using  \cref{lem:frac-max-pointwise-estimate},  we get
%----------------
\begin{align*}
%-------
 \mathcal{M}_{\alpha,b}^{p} (f)(x) &\le C \|b\|_{p,*} \Big(\mathcal{M}^{p}\big(\mathcal{M}_{\alpha}^{p}(f) \big)(x) + \mathcal{M}_{\alpha}^{p}\big(\mathcal{M}^{p}(f) \big)(x)\Big).
%------
\end{align*}
%------------
%Obviously, assertion \labelcref{enumerate:thm-frac-max-var-bmo-2} follows immediately from \labelcref{enumerate:thm-4-chacon2021fractional}  of \cref{lem:frac-max-p-adic-estimate} and \labelcref{inequ:proof-cor-12-1}.

This together with  \cref{lem:frac-max-p-adic-estimate} \labelcref{enumerate:thm-4-chacon2021fractional} and \labelcref{enumerate:thm-5.2-chacon2020variable} indicates that $ \mathcal{M}_{\alpha,b}^{p}$ is bounded from $L^{r(\cdot)}(\mathbb{Q}_{p}^{n})$ to $L^{q(\cdot)}(\mathbb{Q}_{p}^{n})$.

 \labelcref{enumerate:thm-frac-max-var-bmo-3} $\xLongrightarrow{\ \  }$ \labelcref{enumerate:thm-frac-max-var-bmo-4}: For any fixed  $p$-adic ball $B_{\gamma} (x)\subset \mathbb{Q}_{p}^{n}$. By using \cref{lem:lem-3.1-kim2009q}, for all $y\in B_{\gamma} (x)$, we have
%---------------
\begin{align} \label{inequ:pf-3-4-frac-max}
%-------
\begin{split}
%-------
  |b(y)-b_{B_{\gamma}(x)}|  &\le    \dfrac{1}{|B_{\gamma} (x)|_{h}} \dint_{B_{\gamma} (x)} \big|b(y)-b(z) \big| \mathd z    \\
   &=    \dfrac{1}{|B_{\gamma} (x)|_{h}} \dint_{B_{\gamma} (x)} \big|b(y)-b(z) \big| \dchi_{B_{\gamma} (x)}(z) \mathd z    \\
   &\le  \dfrac{1}{|B_{\gamma} (x)|_{h}^{\alpha/n}}   \mathcal{M}_{\alpha,b}^{p} (\dchi_{B_{\gamma} (x)})(y).
%----------
\end{split}
%-------
\end{align}
%------------
 Then, for all $y\in \mathbb{Q}_{p}^{n}$, we get
%---------------
\begin{align*}
%-------
  \big| \big(b(y)-b_{B_{\gamma}(x)} \big)\dchi_{B_{\gamma} (x)}(y) \big|  &\le  |B_{\gamma} (x)|_{h}^{-\alpha/n}    \mathcal{M}_{\alpha,b}^{p} (\dchi_{B_{\gamma} (x)})(y).
%----------
\end{align*}
%------------
Since  $ \mathcal{M}_{\alpha,b}^{p}$ is bounded from $L^{r(\cdot)}(\mathbb{Q}_{p}^{n})$ to $L^{q(\cdot)}(\mathbb{Q}_{p}^{n})$, using    \cref{lem:norm-characteristic-p-adic} \labelcref{enumerate:charact-norm-fraction-p-adic-variable}, we have
%---------------
\begin{align*}
%-------
  \big\| \big(b-b_{B_{\gamma}(x)} \big)\dchi_{B_{\gamma} (x)}  \big\|_{L^{q(\cdot)}(\mathbb{Q}_{p}^{n}) }   &\le  |B_{\gamma} (x)|_{h}^{-\alpha/n}    \big\| \mathcal{M}_{\alpha,b}^{p} (\dchi_{B_{\gamma} (x)}) \big\|_{L^{q(\cdot)}(\mathbb{Q}_{p}^{n}) }     \\
   &\le C |B_{\gamma} (x)|_{h}^{-\alpha/n}    \big\| \dchi_{B_{\gamma} (x)}  \big\|_{L^{r(\cdot)}(\mathbb{Q}_{p}^{n}) }     \\
   &\le C     \|\dchi_{B_{\gamma} (x)}  \|_{L^{q(\cdot)}(\mathbb{Q}_{p}^{n})},
%----------
\end{align*}
%------------
which implies \labelcref{inequ:thm-frac-max-var-bmo-4} since $B_{\gamma} (x)$ is arbitrary and   $C$ is independent of $B_{\gamma} (x)$.

\labelcref{enumerate:thm-frac-max-var-bmo-4} $\xLongrightarrow{\ \  }$  \labelcref{enumerate:thm-frac-max-var-bmo-1}:
  For any   $p$-adic ball $B_{\gamma} (x)\subset \mathbb{Q}_{p}^{n}$,  by using  generalized H\"{o}lder's inequality (see \cref{lem:holder-inequality-p-adic} \labelcref{enumerate:holder-p-adic-variable-1}),  assertion  \labelcref{enumerate:thm-frac-max-var-bmo-4} and    \cref{lem:norm-characteristic-p-adic} \labelcref{enumerate:charact-norm-fraction-p-adic-variable}, we obtain
%-----------------
\begin{align*}
%-------
 \dfrac{1}{|B_{\gamma} (x)|_{h}} \dint_{B_{\gamma} (x)} \big| b(y)-b_{B_{\gamma} (x)}) \big| \mathd y
%  &=   \dfrac{1}{|B_{\gamma} (x)|_{h}} \dint_{B_{\gamma} (x)} \big| b(y)-b_{B_{\gamma} (x)}) \big| \dchi_{B_{\gamma}(x)}(y) \mathd y \\
  &\le \dfrac{C}{|B_{\gamma} (x)|_{h}} \big\| \big(b-b_{B_{\gamma}(x)} \big)\dchi_{B_{\gamma} (x)}  \big\|_{L^{q(\cdot)}(\mathbb{Q}_{p}^{n}) }   \|\dchi_{B_{\gamma} (x)}  \|_{L^{q'(\cdot)}(\mathbb{Q}_{p}^{n})}   \\
%   &=  C \dfrac{\big\| \big(b-b_{B_{\gamma}(x)} \big)\dchi_{B_{\gamma} (x)}  \big\|_{L^{q(\cdot)}(\mathbb{Q}_{p}^{n}) }}{\|\dchi_{B_{\gamma} (x)}  \|_{L^{q(\cdot)}(\mathbb{Q}_{p}^{n})} }       \\
%  &\qquad \times    \dfrac{1}{|B_{\gamma} (x)|_{h}}  \|\dchi_{B_{\gamma} (x)}  \|_{L^{q(\cdot)}(\mathbb{Q}_{p}^{n})} \|\dchi_{B_{\gamma} (x)}  \|_{L^{q'(\cdot)}(\mathbb{Q}_{p}^{n})}    \\
&\le C.
%----------
\end{align*}
%------------

 This shows that $b\in  \bmo(\mathbb{Q}_{p}^{n})$ by   \cref{def.5.3-bmo-space-kim2009carleson} since the constant $C$ is independent of $B_{\gamma} (x)$.

 The proof of \cref{thm:frac-max-var-bmo} is finished.
%--------------------------------
\end{proof}
%--------------

\subsection{ Proofs of \cref{thm:nonlinear-frac-max-bmo-morrey} and \cref{thm:frac-max-bmo-morrey}}

%-------------
\begin{proof}[Proof of \cref{thm:nonlinear-frac-max-bmo-morrey}]
%-------------
%----------------
\begin{enumerate}[label=(T.\arabic*)]
 %%-------------
  \item  1)\ We first prove ``$\Longrightarrow$": Assume  $b\in  \bmo(\mathbb{Q}_{p}^{n})$  and $b^{-}\in L^{\infty}(\mathbb{Q}_{p}^{n})$.
  By \cref{lem:nonlinear-frac-max-pointwise-estimate}  \labelcref{enumerate:thm-nonlinear-frac-max-pointwise-estimate-2}  and  \cref{lem:frac-max-pointwise-estimate}, we have
%----------
\begin{align*}
%------------------
  \big| [b,\mathcal{M}_{\alpha}^{p}](f)(x)  \big| &\le \mathcal{M}_{\alpha,b}^{p} (f)(x)+2b^{-}(x)\mathcal{M}_{\alpha}^{p} (f)(x)    \\
  &\le C \|b\|_{p,*} \Big(\mathcal{M}^{p}\big(\mathcal{M}_{\alpha}^{p}(f) \big)(x) + \mathcal{M}_{\alpha}^{p}\big(\mathcal{M}^{p}(f) \big)(x)\Big) +2b^{-}(x)\mathcal{M}_{\alpha}^{p} (f)(x).
%------------
\end{align*}
%----------
This together with   \cref{lem:lem10-he2022characterization} and \cref{lem:max-function-bound-morrey}  gives that   $ [b,\mathcal{M}_{\alpha}^{p}] $ is bounded from  $L^{r,\lambda}(\mathbb{Q}_{p}^{n})$ to $L^{q,\lambda}(\mathbb{Q}_{p}^{n})$.

 2)\ Now, we prove  ``$\Longleftarrow$": Suppose $ [b,\mathcal{M}_{\alpha}^{p}] $ is bounded from  $L^{r,\lambda}(\mathbb{Q}_{p}^{n})$ to $L^{q,\lambda}(\mathbb{Q}_{p}^{n})$.  With the help of \cref{cor:nonlinear-frac-max-bmo}, we only need to prove that  \eqref{inequ:cor-nonlinear-frac-max-bmo-4} holds, i.e.,  there exists a positive constant $C$ such that
%--------
\begin{align*}
%-----------
  \Big( \dfrac{1}{|B_{\gamma} (x)|_{h}} & \dint_{B_{\gamma} (x)} \big|b(y)-|B_{\gamma} (x)|_{h}^{-\alpha/n} \mathcal{M}_{\alpha,B_{\gamma} (x)}^{p} (b)(y) \big|^{q} \mathd y \Big)^{1/q} \le C.
%-----------------
\end{align*}
%-----------
For any fixed  $p$-adic ball $B_{\gamma} (x)\subset \mathbb{Q}_{p}^{n}$ and any $y\in B_{\gamma} (x)$, it follows from   \cref{lem:frac-max-pointwise-property} \labelcref{enumerate:lem-frac-max-pointwise-property-1} that
%----------
\begin{align*}
%------------------
   \mathcal{M}_{\alpha}^{p}(b\dchi_{B_{\gamma}(x)})(y)   =  \mathcal{M}_{\alpha,B_{\gamma}(x)}^{p}(b)(y)
%------------
 \ \text{and} \
%------------
    \mathcal{M}_{\alpha}^{p}(\dchi_{B_{\gamma}(x)})(y)   =  \mathcal{M}_{\alpha,B_{\gamma}(x)}^{p}(\dchi_{B_{\gamma}(x)})(y)=|B_{\gamma}(x)|_{h}^{\alpha/n}.
%------------
\end{align*}
%----------
Then, using  \cref{def.2.3-morrey-p-adic-ma2020weighted}, we have
%--------
\begin{align*}
%-----------
 \Big( \dfrac{1}{|B_{\gamma} (x)|_{h}} & \dint_{B_{\gamma} (x)} \big|b(y)-|B_{\gamma} (x)|_{h}^{-\alpha/n} \mathcal{M}_{\alpha,B_{\gamma} (x)}^{p} (b)(y) \big|^{q} \mathd y \Big)^{1/q} \\
%   &= \Big( \dfrac{1}{|B_{\gamma} (x)|_{h}^{1+\alpha q/n}} \dint_{B_{\gamma} (x)} \big||B_{\gamma} (x)|_{h}^{\alpha/n} b(y)- \mathcal{M}_{\alpha,B_{\gamma} (x)}^{p} (b)(y) \big|^{q} \mathd y \Big)^{1/q}  \\
%   &= \Big( \dfrac{1}{|B_{\gamma} (x)|_{h}^{1+\alpha q/n}} \dint_{B_{\gamma} (x)} \big|  b(y)\mathcal{M}_{\alpha}^{p}(\dchi_{B_{\gamma}(x)})(y)-  \mathcal{M}_{\alpha}^{p}(b\dchi_{B_{\gamma}(x)})(y) \big|^{q} \mathd y \Big)^{1/q}  \\
   &= \Big( \dfrac{1}{|B_{\gamma} (x)|_{h}^{1+\alpha q/n}} \dint_{B_{\gamma} (x)} \big|  [b,\mathcal{M}_{\alpha}^{p}](\dchi_{B_{\gamma}(x)})(y)  \big|^{q} \mathd y \Big)^{1/q}  \\
%   &= |B_{\gamma} (x)|_{h}^{\frac{\lambda}{qn} -\frac{\alpha}{n}-\frac{1}{q}} \Big( \dfrac{1}{|B_{\gamma} (x)|_{h}^{ \lambda/n}}\dint_{B_{\gamma} (x)} \big|  [b,\mathcal{M}_{\alpha}^{p}](\dchi_{B_{\gamma}(x)})(y)  \big|^{q} \mathd y \Big)^{1/q}  \\
   &\le C|B_{\gamma} (x)|_{h}^{\frac{\lambda}{qn} -\frac{\alpha}{n}-\frac{1}{q}}  \big\|  [b,\mathcal{M}_{\alpha}^{p}](\dchi_{B_{\gamma}(x)})   \big\|_{L^{q,\lambda}(\mathbb{Q}_{p}^{n})}   \\
   &\le C |B_{\gamma} (x)|_{h}^{\frac{\lambda}{qn} -\frac{\alpha}{n}-\frac{1}{q}}   \big\|  \dchi_{B_{\gamma}(x)}    \big\|_{L^{r,\lambda}(\mathbb{Q}_{p}^{n})}   \\
   &\le C,
%-----------------
\end{align*}
%-----------
 where in the last step we have used  $1/q=1/r-\alpha/(n-\lambda)$  and the fact from \cref{def.2.3-morrey-p-adic-ma2020weighted}  and \cref{lem:norm-characteristic-p-adic}  \labelcref{enumerate:charact-norm-p-adic-he}
 %-----------------
\begin{align} \label{inequ:pf-nonlinear-frac-max-morrey}
%-------
  \|\dchi_{B}\|_{L^{r,\lambda}(\mathbb{Q}_{p}^{n})} &\le  \vert B \vert^{(1-\lambda/n)/r}.
%----------
\end{align}
%----------
This completes the proof.
%-----------
   \item  By using a similar reasoning as for \labelcref{enumerate:thm-nonlinear-frac-max-bmo-morrey-1} in \cref{thm:nonlinear-frac-max-bmo-morrey}, we can  deduce the desired result. Hence, we omit the details.                                \qedhere
%-----------
\end{enumerate}
%------------
%-------------
\end{proof}
%--------------

%-------------
\begin{proof}[Proof of \cref{thm:frac-max-bmo-morrey}]
%-------------
Reasoning as the proof of  \cref{thm:nonlinear-frac-max-bmo-morrey}.
%----------------
\begin{enumerate}[label=(T.\arabic*)]
 %%-------------
  \item  1)\  For ``$\Longrightarrow$": Since $b\in  \bmo(\mathbb{Q}_{p}^{n})$, then, using  \cref{lem:frac-max-pointwise-estimate}, and combining \cref{lem:lem10-he2022characterization} and \cref{lem:max-function-bound-morrey},   we can obtain the result we need.
%%----------
%\begin{align*}
%%------------------
%  \mathcal{M}_{\alpha,b}^{p} (f)(x)  &\le C \|b\|_{p,*} \Big(\mathcal{M}^{p}\big(\mathcal{M}_{\alpha}^{p}(f) \big)(x) + \mathcal{M}_{\alpha}^{p}\big(\mathcal{M}^{p}(f) \big)(x)\Big)
%%------------
%\end{align*}
%%----------
%This together with   \cref{lem:lem10-he2022characterization} and \cref{lem:max-function-bound-morrey}  gives that   $\mathcal{M}_{\alpha,b}^{p}$ is bounded from  $L^{r,\lambda}(\mathbb{Q}_{p}^{n})$ to $L^{q,\lambda}(\mathbb{Q}_{p}^{n})$.

 2)\ For  ``$\Longleftarrow$": Let $\mathcal{M}_{\alpha,b}^{p}  $ is bounded from  $L^{r,\lambda}(\mathbb{Q}_{p}^{n})$ to $L^{q,\lambda}(\mathbb{Q}_{p}^{n})$.  According to \cref{cor:frac-max-bmo}, we  just  need to  verify that  \eqref{inequ:cor-frac-max-bmo-4} is  valid, i.e.,  there exists a positive constant $C$ such that
%--------
\begin{align*}
%-----------
  \Big( \dfrac{1}{|B_{\gamma} (x)|_{h}} & \dint_{B_{\gamma} (x)} \big|b(y)-b_{B_{\gamma}(x)} \big|^{q} \mathd y \Big)^{1/q} \le C.
%-----------------
\end{align*}
%-----------
For any   $p$-adic ball $B_{\gamma} (x)\subset \mathbb{Q}_{p}^{n}$ and any $y\in B_{\gamma} (x)$,  using \eqref{inequ:pf-3-4-frac-max},  \cref{def.2.3-morrey-p-adic-ma2020weighted} and \eqref{inequ:pf-nonlinear-frac-max-morrey}, we get
%--------
\begin{align*}
%-----------
 \Big( \dfrac{1}{|B_{\gamma} (x)|_{h}}   \dint_{B_{\gamma} (x)} \big|b(y)-b_{B_{\gamma}(x)} \big|^{q} \mathd y \Big)^{1/q}
 &\le \Big( \dfrac{1}{|B_{\gamma} (x)|_{h}^{1+\alpha q/n}} \dint_{B_{\gamma} (x)} \Big( \mathcal{M}_{\alpha,b}^{p} (\dchi_{B_{\gamma} (x)})(y) \Big)^{q} \mathd y \Big)^{1/q}  \\
%   &= |B_{\gamma} (x)|_{h}^{\frac{\lambda}{qn} -\frac{\alpha}{n}-\frac{1}{q}} \Big( \dfrac{1}{|B_{\gamma} (x)|_{h}^{ \lambda/n}}\dint_{B_{\gamma} (x)} \big|  \mathcal{M}_{\alpha,b}^{p} (\dchi_{B_{\gamma} (x)})(y)  \big|^{q} \mathd y \Big)^{1/q}  \\
   &\le C|B_{\gamma} (x)|_{h}^{\frac{\lambda}{qn} -\frac{\alpha}{n}-\frac{1}{q}}  \big\| \mathcal{M}_{\alpha,b}^{p}(\dchi_{B_{\gamma}(x)})   \big\|_{L^{q,\lambda}(\mathbb{Q}_{p}^{n})}   \\
   &\le C |B_{\gamma} (x)|_{h}^{\frac{\lambda}{qn} -\frac{\alpha}{n}-\frac{1}{q}}   \big\|  \dchi_{B_{\gamma}(x)}    \big\|_{L^{r,\lambda}(\mathbb{Q}_{p}^{n})}   \\
   &\le C.
%-----------------
\end{align*}
%-----------
%-----------
   \item  By applying a similar  argument to   \labelcref{enumerate:thm-frac-max-bmo-morrey-1} in \cref{thm:frac-max-bmo-morrey}, we can  achieve the desired
result.                                \qedhere
%-----------
\end{enumerate}
%------------
%-------------
\end{proof}
%--------------

%-----------------------
% \subsubsection*{Acknowledgments:}
%The authors cordially  thank the anonymous referees who gave valuable  suggestions and useful comments which have lead to the improvement of this paper.

%-----------------------
 \subsubsection*{Funding information:}
 This work was partly supported by Project of Heilongjiang Province Science and Technology Program (No.2019-KYYWF-0909), the National Natural Science Foundation of China (No.11571160),the Reform and Development Foundation for Local Colleges and Universities of the Central Government(No.2020YQ07) and the Scientific Research Fund of Mudanjiang Normal University (No.D211220637).

\phantomsection
\addcontentsline{toc}{section}{References}
%----------
%\bibliographystyle{tugboat}  %{amsplain} %
%\bibliography{wu-reference}

%============================
\end{document}